\documentclass[reqno]{amsart}
\pdfoutput=1
\usepackage{dsfont}
  \usepackage{paralist}
  \usepackage{graphics}
  \usepackage{epsfig}
\usepackage{graphicx}  
 \usepackage{tikz}
 \usepackage{todonotes}
\usepackage{stmaryrd}
\usepackage{epstopdf}

\newtheorem{theorem}{Theorem}[section]

\newtheorem{lemma}[theorem]{Lemma}

\theoremstyle{definition}
\newtheorem{definition}[theorem]{Definition}
\newtheorem{remark}{Remark}

\newcommand{\di}[1]{\,\mathrm{d}#1}
\newcommand{\jump}[1]{\llbracket #1\rrbracket}

\newcommand{\R}{\mathbb{R}}
\newcommand{\A}{\mathbb{A}}

\newcommand{\Q}{\mathbb{Q}}

\newcommand{\N}{\mathbb{N}}
\newcommand{\Z}{\mathbb{Z}}

\newcommand{\CC}{\mathcal{C}}
\newcommand{\per}{{\overline{\#}}}
\newcommand{\dive}{\operatorname{div}}

\newcommand{\Sym}{\operatorname{Sym}}

\newcommand{\twosc}{\stackrel{2}{\rightarrow}}
\newcommand{\stwosc}{\stackrel{2-str.}{\longrightarrow}}
\newcommand{\e}{\varepsilon}
\newcommand{\p}{\varphi}

\newcommand{\dx}{\,\mathrm{d}x}
\newcommand{\ds}{\,\mathrm{d}s}
\newcommand{\dt}{\,\mathrm{d}t}

\newcommand{\ddt}{\frac{\operatorname{d}}{\mathrm{d}t}}

\newcommand{\dy}{\,\mathrm{d}y}
\newcommand{\dtau}{\,\mathrm{d}\tau}

\newcommand{\exe}{\left(x,\frac{x}{\e}\right)}

\title[Homogenization of a Thermoelasticity Problem] 
      {Homogenization of a Fully Coupled Thermoelasticity Problem for a Highly Heterogeneous Medium With {\em a Priori} Known Phase Transformations}

\author[Michael Eden, Adrian Muntean]{}


 \email{leachim@math.uni-bremen.de}
 \email{adrian.muntean@kau.se}

\begin{document}
\maketitle

\centerline{\scshape Michael Eden}
\medskip
{\footnotesize
 \centerline{Center for Industrial Mathematics, FB 3}
   \centerline{University of Bremen, Germany}
   \centerline{ Bibliotheksstr.~1, 28359, Bremen}
} 

\medskip

\centerline{\scshape Adrian Muntean
}
\medskip
{\footnotesize
 \centerline{Department of Mathematics and Computer Science}
   \centerline{University of Karlstad, Sweden}
   \centerline{Universitetsgatan 2, 651 88 Karlstad}
}

\bigskip

\begin{abstract}
We investigate a linear, fully coupled thermoelasticity problem for a highly heterogeneous, two-phase medium.
The medium in question consists of a connected matrix with disconnected, initially periodically distributed inclusions separated by a sharp interface undergoing an {\em a priori} known interface movement due to phase transformations.
After transforming the moving geometry to an $\e$-periodic, fixed reference domain, we establish the well-posedness of the model and derive a number of $\e$-independent {\em a priori} estimates.
Via a two-scale convergence argument, we then show that the $\e$-dependent solutions converge to solutions of a corresponding upscaled model with distributed time-dependent microstructures.
\end{abstract}

\section{Introduction}
In this paper, we consider a heterogeneous medium where the two building  components are different solid phases of the same material (like  \emph{austenite} and \emph{bainite} phases in steel, e.g.) separated by a sharp interface. 
One phase is assumed to be a connected matrix in which finely interwoven, periodically distributed inclusion of the second phase (which is therefore disconnected) are embedded. We refer to these phases as microstructures. 

Our interest is the case where phase transformations are possible, e.g., one phase might grow at the expanse of the other phase, thereby leading to an interface movement and, as a consequence, time dependent domains that are not necessarily periodic anymore.
However, we assume to have {\em a priori} knowledge of the phase transformation, i.e., the movement of the interface is prescribed.
For a rather general modeling of phase transformations (including a possible mathematical treatment), we refer the reader  to~\cite{V96}, and for the metallurgical perspective on phase transformation in steel (especially, w.r.t~the bainite transformation), we refer to~\cite{F13, PE15, S08}.
Looking at such a highly-heterogeneous medium, we study  the coupling between the mechanics of the material and the thermal conduction effect (\emph{thermomechanics}) under the influence of the phase transformation.
In particular, we explore the interplay between \emph{surface stresses} and \emph{latent heat}, see for instance~\cite{K79} for related  thermoelasticity scenarios.
In this work, we assume the quasi-static assumption to hold; that is the mechanical processes are reversible.
Furthermore, the constitutive laws are taken to be linear.
Our main contribution here is the treatment of the mechanical dissipation and of {\em a priori} prescribed phase transformations in the thermoelasticity setting.

It is worth noting that the homogenization of different thermoelasticity problems has already been addressed in the  literature.
In one of the earlier works,~\cite{F83}, a one-phase linear thermoelasticity problem is homogenized via a semi-group approach.
In~\cite{T11}, a formal homogenization via asymptotic expansion for a similar model (but for a one-dimensional geometry) was conducted.
A two-phase problem including transmission conditions and discontinuities at the interface has been investigated in the context of homogenization (using \emph{periodic unfolding}) in~\cite{ETT15}.
A similar situation of a highly heterogeneous two-phase medium with {\em a priori} given phase transformation was considered in~\cite{EKK02}.
Here, the authors use formal asymptotic expansions to derive a homogenized model.
We also want to point out the structural similarity between the thermoelasticity models and models for poroelasticity, cf.~\emph{Biot's linear poroelasticity}~\cite{B41, S02}; for a reference of the derivation of the Biot model via two-scale homogenization, we refer to~\cite[Section~5.2]{M03}.
Examples for homogenization in the context of two-phase poroelasticity, so called double poroelasticity, can be found in~\cite{A13,EB14}.
For some homogenization results via formal asymptotics for problems where the micro-structural changes are not prescribed, we refer to~\cite{BPR16,KNP14, M08b}.

As an alternative approach in the modeling of phase transformation, in particular in the case of phase transformation in steel, phase-field models are often considered, we refer to, e.g., \cite{MSA15,MBW08}.
Some thoughts regarding possible numerical simulations of a similar one-phase problem for a highly heterogeneous media are given in~\cite{Pl93}.
In~\cite{T11}, a numerical framework based on homogenization (via averaging) for a thermoelasticity problem in highly heterogeneous media is developed and investigated.

The paper is organized as follows: 
In Section~\ref{section:setting}, we introduce the $\e$-microscopic geometry and the thermoelasticity problem and, then, transform this to a fixed reference domain. The well-posedness of our microscopic model is investigated in Section~\ref{section:analysis}.
In addition, $\e$-independent estimates necessary for the homogenization process are established.
Finally, in Section~\ref{section:homogenization}, we perform the homogenization procedure relying on the two-scale convergence technique.

\section{Setting and transformation to fixed domain}\label{section:setting}

We start by describing the geometrical setting of the $\e$-parametrized microscopic problem including the transformation characterizing the interface movement.
After that, we go on with formulating the microscopic problem for a highly heterogeneous media -- first for the moving interface and then for the back-transformed, fixed interface.  

We note that our setting (with the transformation) is closely related to the notion of \emph{locally periodic} domains, see~\cite{FAZM11,VM11}.
In addition, we also refer to~\cite{D15,M08}, where similar transformation settings are introduced.

Let $S=(0,T)$, $T>0$, be a time interval. Let $\Omega$ be the interior of a union of a finite number of closed cubes $Q_j$, $1\leq j\leq n$, $n\in\N$, whose vertices are in $\Q^3$ such that, in addition, $\Omega$ is a Lipschitz domain. 
In general, $\Omega$ could be much more general.
By this particular choice, we avoid the inherent technical difficulties that would arise in the homogenization process due to the involvement of general geometries; we are focusing instead on the technical difficulties arising a) due to the strong coupling in the structure of the governing partial differential equations and b) due to the time-dependency of the geometry. 

In addition, we denote the outer normal vector of $\Omega$ with $\nu=\nu(x)$.
Let $Y=(0,1)^3$ be the open unit cell in $\R^3$.
Take $Y_A,$ $Y_B\subset Y$  two disjoint open sets, such that $Y_A$ is connected, such that $\Gamma:=\overline{Y_A}\cap\overline{Y_B}$ is a $C^3$ interface, $\Gamma=\partial Y_B$, $\overline{Y_B}\subset Y$, and $Y=Y_A\cup Y_B\cup \Gamma$, see Figure~\ref{s:fig}.
With $n_0=n_0(y)$, $y\in\Gamma$, we denote the normal vector of $\Gamma$ pointing outwards of $Y_B$.

For $\e>0$, we introduce the $\e Y$-periodic, initial domains $\Omega_A^\e$ and $\Omega_B^\e$ and interface $\Gamma^\e$ representing the two phases and the phase boundary, respectively, via ($i\in\{A,B\}$)
\begin{align*}
	\Omega^\e_i=\Omega\cap\left(\bigcup_{k\in\Z^3}\e(Y_i+k)\right),\qquad
	\Gamma^\e=\Omega\cap\left(\bigcup_{k\in\Z^3}\e(\Gamma+k)\right).
\end{align*}
Here, for a set $M\subset\R^3$, $k\in\Z^3$, and $\e>0$, we employ the notation
$$
	\e(M+k):=\left\{x\in\R^3\ : \ \frac{x}{\e}-k\in M\right\}.
$$
From now on, we take $\e=(\e_n)_{n\in\N}$ to be a sequence of monotonically decreasing positive numbers converging to zero such that $\Omega$ can be represented as the union of cubes of size $\e$.
Note that this is possible due to the assumed structure of $\Omega$.

Here $n_0^\e=n_0(\frac{x}{\e})$, $x\in\Gamma^\e$, denotes the unit normal vector (extended by periodicity) pointing outwards $\Omega_B^\e$ into $\Omega_A^\e$.
The above construction ensures that  $\Omega_A^\e$ is connected and that $\Omega_B^\e$ is disconnected.
We also have that $\partial\Omega_B^\e\cap\partial\Omega=\emptyset$.
In the different case that both $\Omega_A^\e$ and $\Omega_B^\e$ are connected, we additionally would need to rely on special uniform extension operators, see~\cite{HB14}, in order to pass to the homogenization limit.

We assume that $s\colon\overline{S}\times\overline\Omega\times\R^3\to\overline{Y}$ is a function such that
\begin{enumerate}
	\item $s\in C^1(\overline{S};C^2(\overline{\Omega})\times C^2_\#(Y))$,\footnote{%
	Here, and in the following, the $\#$ subscript denotes periodicity, i.e., for $k\in\N_0$, we have $C^k_\#(Y)=\{f\in C^k(\R^3): f(x+e_i)=f(x)\ \text{for all} \ x\in\R^3\}$, $e_i$ basis vector of $\R^3$.}
	\item $s(t,x,\cdot)_{|\overline{Y}}\colon\overline{Y}\to\overline{Y}$ is bijective for every $(t,x)\in\overline{S}\times\overline{\Omega}$,
	\item $s^{-1}\in C^1(\overline{S};C^2(\overline{\Omega})\times C^2_\#(Y))$,\footnote{Here, $s^{-1}\colon\overline{S}\times\overline\Omega\times\R^3\to\overline{Y}$ is the unique function such that $s(t,x,s^{-1}(t,x,y))=y$ for all $(t,x,y)\in\overline{S}\times\overline{\Omega}\in\overline{Y}$ extended by periodicity to all $y\in\R^3$.}
	\item $s(0,x,y)=y$ for all $x\in\overline{\Omega}$ and all $y\in\overline Y$,
	\item $s(t,x,y)=y$ for all $(t,x)\in\overline{S}\times\overline{\Omega}$ and all $y\in\partial Y$,
	\item there is a constant $c>0$ such that $\mathrm{dist}(\partial Y,\gamma)>c$ for all $\gamma\in s(t,x,\Gamma)$ and $(t,x)\in\overline{S}\times\overline{\Omega}$,
	\item $s(t,x,y)=y$ for all $(t,x)\in\overline{S}\times\overline{\Omega}$ and for all $y\in Y$ such that $\mathrm{dist}(\partial Y,y)<\frac{c}{2}$,
	\item there are constants $c_s,C_s>0$ such that
	      $$
					c_s\leq\det(\nabla s(t,x,y))\leq C_s,\quad (t,x,y)\in\overline{S}\times\overline\Omega\times\R^3
				$$
\end{enumerate}
and set the $(t,x)$-parametrized sets
\begin{align*}
	Y_A(t,x)=s(t,x,Y_A),\quad
	Y_B(t,x)=s(t,x,Y_B),\quad
	\Gamma(t,x)=s(t,x,\Gamma).
\end{align*}
Here, Assumptions~(1)-(3) ensure that the transformation is regular enough for our further considerations (e.g., to guarantee that the curvature of the deformed domains is well-defined). 
For the initial configuration, we have $Y_A=Y_A(0,x)$, $Y_B=Y_B(0,x)$, and $\Gamma=\Gamma(0,x)$ (Assumption~(4)).
In addition, with Assumption~(6), we get a uniform (w.r.t.~$(t,x)\in\overline{S}\times\overline{\Omega}$) minimum distance between the interface $\Gamma(t,x)$ and the boundary of the $Y$-cell and, with Assumption~(5) and Assumption~(7), make sure that points near the boundary of the unit cell $Y$ are not deformed.
Finally, Assumption~(8) is of particular importance when it comes to proving $\e$-independent estimates.

We introduce the operations
\begin{align*}
	[\cdot]&\colon\R^3\to\Z^3,\quad [x]=k \ \text{such that}\  x-[x]\in Y,\\
	\{\cdot\}&\colon\R^3\to Y,\quad \{x\}=x-[x]
\end{align*}
and define the $\e$-dependent function\footnote{This is the typical notation in the context of homogenization via the \emph{periodic unfolding method}, see, e.g., \cite{CDG08, D12}.}
$$
	s^\e\colon\overline{S}\times\overline{\Omega}\to\R^3,\quad s^\e(t,x):=\e\left[\frac{x}{\e}\right]+\e s\left(t,\e\left[\frac{x}{\e}\right],\left\{\frac{x}{\e}\right\}\right).
$$
This function is well-defined as $\left\{\frac{x}{\e}\right\}\in Y$ and $\e\left[\frac{x}{\e}\right]\in\overline{\Omega}$.
Since $s\left(t,x,y\right)=y$ for all $(t,x)\in\overline{S}\times\overline{\Omega}$ and for all $y\in Y$ such that $\mathrm{dist}(\partial Y,y)>\frac{c}{2}$, we see that 
$$
	s^\e\in C^1(\overline{S};C^2(\overline{\Omega})).
$$
For $i\in\{A,B\}$ and $t\in\overline{S}$, we set the time dependent sets $\Omega_i^\e(t)$ and $\Gamma^\e(t)$ and the corresponding non-cylindrical space-time domains $Q_i^\e$ and space-time phase boundary $\Sigma^\e$ via
\begin{alignat*}{2}
	\Omega_i^\e(t)&=s^\e(t,\Omega_i^\e),&\qquad Q_i^\e&=\bigcup_{t\in S}\left(\{t\}\times\Omega_i^\e(t)\right),\\
	\Gamma^\e(t)&=s^\e(t,\Gamma^\e),&\qquad\Sigma^\e&=\bigcup_{t\in S}\left(\{t\}\times\Gamma^\e(t)\right),
\end{alignat*}
and denote by $n^\e=n^\e(t,x)$, $t\in S$, $x\in\Gamma^\e(t)$, the unit normal vector pointing outwards $\Omega_B^\e(t)$ into $\Omega_A^\e(t)$.
The time-dependent domains $\Omega_i^\e(t)$ host the phases at time $t\in\overline{S}$ and model the movement of the interface $\Gamma^\e$.
We emphasize that, for any $t>0$, the sets $\Omega_A^\e(t)$, $\Omega_B^\e(t)$, and $\Gamma^\e(t)$ do not need to be periodic.

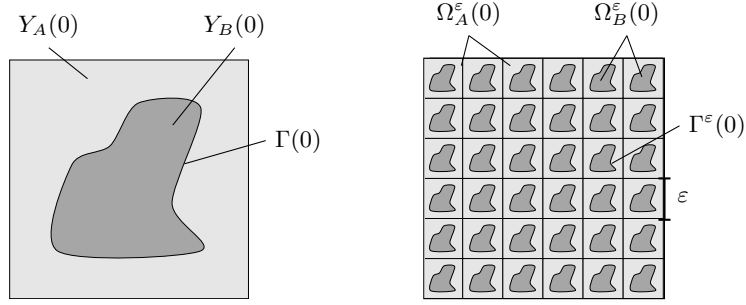
\begin{figure}[!ht]
\hspace{1cm}
\begin{tikzpicture}[scale=1.03]
	\pgftext{\includegraphics[width=0.25\textwidth,]{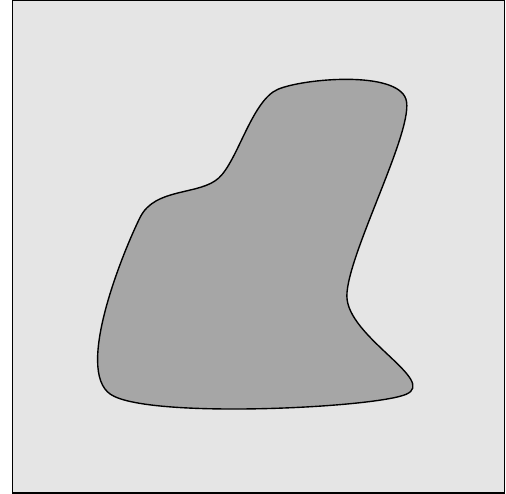}};
	\draw[thin] (0.6,0.7) -- (1.4,1.7);
	\draw (1.4,1.7) node[above] {{\small{$Y_B(0)$}}};
	\draw[thin] (-0.5,1.3) -- (-1,1.7);
	\draw (-1,1.7) node[above] {{\small{$Y_A(0)$}}};
	\draw[thin] (0.75,0.2) -- (1.8,0.5);
	\draw (1.8,0.5) node[right] {{\small{$\Gamma(0)$}}};
\end{tikzpicture}
\hspace{1cm}
\begin{tikzpicture}[scale=1]
	\pgftext{\includegraphics[width=0.25\textwidth]{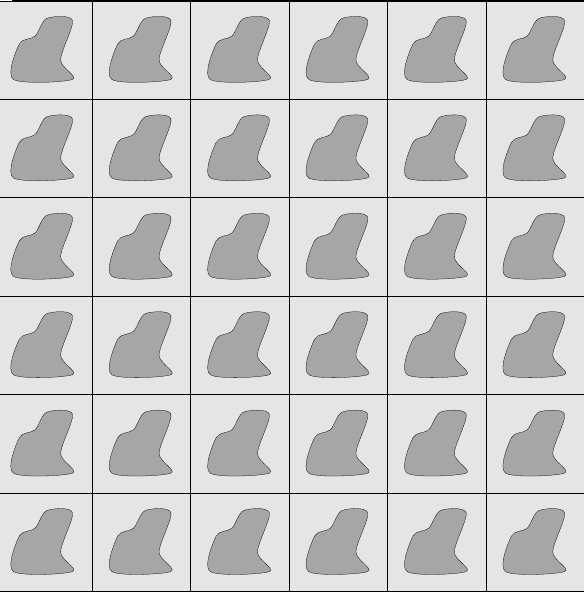}};
	\draw[thin] (-1.6,-1.6) rectangle (1.6,1.6);
	\draw[thin] (0.75,1.3) -- (1.1,1.9);
	\draw[thin] (1.3,1.35) -- (1.1,1.9);
	\draw (1.1,1.9) node[above] {{\small{$\Omega_B^\e(0)$}}};
	\draw[thin] (-0.45,1.5) -- (-1,1.9);
	\draw[thin] (-1.1,1.5) -- (-1,1.9);
	\draw (-1,1.9) node[above] {{\small{$\Omega_A^\e(0)$}}};
	\draw[thin] (0.9,0.2) -- (1.8,0.7);
	\draw (1.8,0.7) node[right] {{\small{$\Gamma^\e(0)$}}};
	\draw[thick] (1.52,0) -- (1.68,0);
	\draw[thick] (1.52,-0.55) -- (1.68,-0.55);
	\draw[thick] (1.6,0) -- (1.6,-0.55);
	\draw (1.65,-0.25) node[right] {{\small{$\e$}}};
\end{tikzpicture}
\caption{Reference geometry and the resulting $\e$-periodic initial configuration. Note that for $t\neq0$, these domains typically loose their periodicity.}
\label{s:fig}
\end{figure}

For all $(t,x)\in \overline{S}\times\overline{\Omega}$, we introduce the functions
\begin{subequations}\label{s:trafo_quantities}
\begin{alignat}{2}
	F^\e&\colon\overline{S}\times\overline{\Omega}\to\R^{3\times3},\qquad &
		F^\e(t,x)&:=\nabla s^\e(t,x),\label{s:trafo_quantities:1}\\
	J^\e&\colon\overline{S}\times\overline{\Omega}\to\R,\qquad&
		J^\e(t,x)&:=\det\left(\nabla s^\e(t,x)\right),\\
	v^\e&\colon\overline{S}\times\overline{\Omega}\to\R^{3},\qquad&
		v^\e(t,x)&:=\partial_ts^\e(t,x),
\end{alignat}
and see that\footnote{Here, $(F^{\e})^{-T}=\left((F^{\e})^{-1}\right)^T$.}
\begin{equation}
	n^\e(t,s^\e(t,x))=\frac{(F^{\e})^{-T}(t,x)n_0^\e(x)}{\left|(F^{\e})^{-T}(t,x)n_0^\e(x)\right|}.
\end{equation}
In addition, we need the \emph{mean curvature}  $H_\Gamma^\e$  and \emph{the normal velocity} (inwards $\Omega_A(t)$)  $W_\Gamma^\e$ of the interface $\Gamma^\e(t)$ (w.r.t.~to the coordinates of the initial configuration!):
\begin{alignat}{2}
	W_\Gamma^\e&\colon\overline{S}\times\Gamma^\e\to\R,\quad&
		W_\Gamma^\e(t,x)&:=v^\e(t,x)\cdot n^\e(t,s^\e(t,x)),\\
	H_\Gamma^\e&\colon\overline{S}\times\Gamma^\e\to\R,\quad&
		H_\Gamma^\e(t,x)&:=-\dive\left((F^{\e})^{-1}(t,x)n^\e(t,s^\e(t,x))\right).\label{s:trafo_quantities:6}
\end{alignat}
\end{subequations}
We note that, via this definition, $H_\Gamma^\e$ is non positive at points $x\in\Gamma^\e$, where the intersection of $\Omega_B^\e$ and a sufficiently small ball with center $x$ is a convex set, and non negative when this holds true for $\Omega_A^\e$.

Under this given transformation describing the phase transformation, i.e., the function $s^\e$ and the resulting time dependent domains $\Omega_i^\e$, we consider a fully coupled thermoelasticity problem where we assume the mechanical response to be quasi-stationary and the constitutive laws to be linear.

For $i\in\{A,B\}$, $t\in S$, and $x\in\Omega_i^\e(t)$, let $u_i=u_i(t,x)$ denote the deformation and $\theta_i=\theta_i(t,x)$ the temperature in the respective phase.

The bulk equations are given as (for more details regarding the modeling, we refer to~\cite{B56,K79})
\begin{subequations}\label{p:full_problem_moving}
\begin{alignat}{2}
	-\dive(\CC_A^\e e(u_A^\e)-\alpha_A^\e\theta_A\mathds{I}_3)&=f_{u_A}^\e&\quad&\text{in}\ \ Q_A^\e,\label{p:full_problem_moving:1}\\
	-\dive(\CC_B^\e e(u_B^\e)-\alpha_B^\e\theta_B\mathds{I}_3)&=f_{u_B}^\e&\quad&\text{in}\ \ Q_B^\e,\label{p:full_problem_moving:2}\\
	\partial_t\left(\rho_Ac_{dA}\theta_A^\e+\gamma_A^\e\dive u_A^\e\right)-\dive(K_A^\e\nabla\theta_A^\e)&=f_{\theta_A}^\e&\quad&\text{in}\ \ Q_A^\e,\label{p:full_problem_moving:3}\\
	\partial_t\left(\rho_Bc_{dB}\theta_B^\e+\gamma_B^\e\dive u_B^\e\right)-\dive(K_B^\e\nabla\theta_B^\e)&=f_{\theta_B}^\e&\quad&\text{in}\ \ Q_B^\e.\label{p:full_problem_moving:4}
\end{alignat}
Here, $\CC_i^\e\in\R^{3\times3\times3\times3}$ are the \emph{stiffness} tensors, $\alpha_i^\e>0$ the \emph{thermal expansion} coefficients, $\rho_i>0$ the \emph{mass densities}, $c_{di}>0$ the \emph{heat capacities}, $\gamma_i^\e>0$ are the \emph{dissipation coefficients}, $K_i^\e\in\R^{3\times3}$ the \emph{thermal conductivities}, and $f_{u_i}^\e$, $f_{\theta_i}^\e$ are volume densities.
In addition, $e(v)=1/2(\nabla v+\nabla v^T)$ denotes the linearized strain tensor and $\mathds{I}_3$ the identity matrix.

At the interface between the phases, the transmission of both the temperature and deformation is assumed to be continuous,\footnote{These conditions are sometimes called \emph{coherent} and \emph{homothermal}, see~\cite{BM05}.} i.e.,
\begin{equation}
	\jump{u^\e}=0,\quad \jump{\theta^\e}=0\quad\text{on}\ \ \Sigma^\e,\label{p:full_problem_moving:5}
\end{equation}
where $\jump{v}:=v_A-v_B$ denotes the jump across the boundary separating phase $A$ from phase $B$.

The jump in the flux of force densities across the interface is assumed to be proportional to the mean curvature of the interface leading to
\begin{equation}\label{p:full_problem_moving:6}
	\jump{\CC^\e\e(u^\e)-\alpha^\e\theta^\e\mathds{I}_3}n^\e=-\e^2\widetilde{H_\Gamma^\e}\sigma_0n^\e\quad\text{on}\ \ \Sigma^\e,
\end{equation}
where $\sigma_0>0$ is the coefficient of surface tension and where $\widetilde{H_\Gamma^\e}$ is the mean curvature of the interface w.r.t.~moving coordinates.\footnote{I.e., $\widetilde{H_\Gamma^\e}(t,s^\e(t,x))=H_\Gamma^\e(t,x)$.}
Here, the scaling via $\e^2$ counters the effects of both the interface surface area, note that $\e|\Gamma^\e|\in\mathcal{O}(1)$, and the curvature itself, note that $\e|\widetilde{H_\Gamma^\e}|\in\mathcal{O}(1)$.

In a similar way, the jump of the heat across the interface is assumed to be given via the \emph{latent heat} $L_{AB}\in\R$:
\begin{equation}\label{p:full_problem_moving:7}
	\jump{\rho c_d}\theta^\e \widetilde{W_{\Gamma}^\e}+\jump{\gamma^\e\dive u^\e}\widetilde{W_\Gamma^\e}-\jump{K^\e\nabla\theta^\e}\cdot n^\e= L_{AB}\widetilde{W_\Gamma^\e}\quad\text{in}\ \ \Sigma^\e,
\end{equation}
where $\widetilde{W_\Gamma^\e}$ denotes the normal velocity of the interface w.r.t.~moving coordinates. Note that, if we neglect the dissipation and if we have equal densities and heat capacities in both phases (or, a bit more general, $\jump{\rho c_d}=0$), equation~\eqref{p:full_problem_moving:7} reduces to the usual \emph{Stefan condition}.
More complex interface conditions then equations~\eqref{p:full_problem_moving:6}, \eqref{p:full_problem_moving:6} would arise, if the interface were allowed to be thermodynamically active thereby requiring us to formulate seperate balance equations for surface stress and surface heat, we refer to~\cite{WB15}.

Finally, we pose homogeneous \emph{Dirichlet conditions} for the momentum equation and homogeneous \emph{Neumann conditions} for the heat equation as well as initial conditions for the temperature:
\begin{alignat}{2}
	u_A^\e&=0&\quad&\text{on}\ \ S\times\partial\Omega_A^\e,\\
	-K_A^\e\nabla\theta_A^\e\cdot\nu&=0&\quad&\text{on}\ \ S\times\partial\Omega_A^\e,\\
	\theta^\e(0)&=\theta_{0}^\e&\quad&\text{on}\ \ \Omega,\label{p:full_problem_moving:9}
\end{alignat}
where $\theta_{0}^\e$ is some (possibly highly heterogeneous) initial temperature distribution.
\end{subequations}

To summarize, we are considering a highly heterogeneous medium that is composed of two different phases/microstructures where one phase is a connected matrix in which small inclusions of the other phase are (in the beginning, periodically) embedded (see Figure~\ref{s:fig}), e.g., \emph{bainitic} inclusions in \emph{austenite} steel. 
Due to phase transformations (in our example, the bainitic inclusions might grow at the cost of the austenite phase) which are assumed to be completely known {\em a priori}, the phase domains change with time.
In this geometrical setting, we then investigate the thermomechanical response of the two-phase medium to the surface stresses exerted by the phase interface due to its curvature (equation~\eqref{p:full_problem_moving:6}) and the latent heat released via the phase transformation (equation~\eqref{p:full_problem_moving:7}).
We note that this situation has (in spirit) some similarity with the one considered in~\cite{EKK02}.

Now, we choose a particular scaling (of some coefficients) with respect to the $\e$-parameter that leads to a distributed microstructure model in the homogenization limit:
For $i\in\{A;B\}$, we assume that there are constants $\CC_i\in\R^{3\times3\times3\times3}$, $K_i\in\R^{3\times3}$, $\alpha_i,\gamma_i>0$
such that
\begin{alignat*}{4}
	\CC_A^\e&=\CC_A,&\quad K_A^\e&=K_A,&\quad\alpha_A^\e&=\alpha_A,&\quad\gamma_A^\e&=\gamma_A,\\
	\CC_B^\e&=\e^2\CC_B,&\quad K_B^\e&=\e^2K_B,&\quad\alpha_B^\e&=\e\alpha_B,&\quad\gamma_B^\e&=\e\gamma_B.
\end{alignat*}
These specific $\e$-scalings are quite common in the modeling of two-phase media, see, e.g.,~\cite{A13, CS99, EB14, FAZM11, Y11},
and are usually justified (albeit only heuristically) by assuming different orders of magnitude of the characteristic time scales of the involved physical processes in the respective domains.
In our case, this means that the effect of heat conduction, the dissipation, the stresses, and the thermal expansion are assumed to be smaller/slower in the inclusions when compared to what happens in the matrix.
By two-scale convergence results, this scaling leads to a distributed microstructure model, cf.~\cite[Proposition 1.14.~(ii)]{Al92}.
Some other $\e$-scalings are, of course, possible and, depending on the underlying assumptions regarding the orders of magnitude of the involved processes, might be sensible.
Without the scalings in the bulk equations (i.e., for $\CC_B^\e$, $K_B^\e$, $\alpha_B^\e$, and $\gamma_B^\e$)), e.g., we would expect to get a purely macroscopical limit problem, where only some of the information of the microstructure (and their changes) are coded into the averaged coefficients, similar to the results in~\cite{A11}.
Related problems withno scaling of $\CC_B^\e$ but otherwise the same scaling for similar scenarios (in the context of double poroelasticity) were investigated in~\cite{A13,EB14}.
For a more holistic approach to different sets of scalings and their effect on the homogenization procedure, we refer to~\cite{PB08}.

We assume that the tensors $\CC_i$ and matrices $K_i$ are symmetric and have constant entries and also that there is a constant $c>0$ such that
$\CC_iM:M\geq c\left|M\right|^2$ for all symmetric matrices $M\in\R^{3\times3}$ and such that $K_iv\cdot v\geq c\left|v\right|^2$ for all $v\in\R^3$.
Note that it would also be possible to treat non-constant coefficients as long as estimates~\eqref{lemma:eq:estimates_movement:1}-\eqref{lemma:eq:estimates_movement:5} hold uniformly in time and space and as long as the functions are sufficiently regular for the analysis part to hold.\footnote{E.g., we would need $\rho_ic_{d_i}$, $\CC_i$, and $\alpha_i$ to be differentiable w.r.t.~time in order for Lemma~\ref{lemma:time_regularity} to hold.}

Now, from the construction and the regularity of $s$, we have the following estimates at hand for the transformation-related quantities defined via equations~\eqref{s:trafo_quantities:1}-\eqref{s:trafo_quantities:6}:
\begin{multline}\label{s:estimate_movement}
	\left\|F^\e\right\|_{L^\infty(S\times\Omega)^{3\times3}}+\left\|(F^\e)^{-1}\right\|_{L^\infty(S\times\Omega)^{3\times3}}+\left\|J^\e\right\|_{L^\infty(S\times\Omega)}\\
	+\e^{-1}\left\|v^\e\right\|_{L^\infty(S\times\Omega)^3}
	+\e^{-1}\left\|W_\Gamma^\e\right\|_{L^\infty(S\times\Gamma^\e)}+\e\left\|H_\Gamma^\e\right\|_{L^\infty(S\times\Gamma^\e)}\leq C,
\end{multline}
where $C$ is independent of $\e>0$.
Furthermore, we also see that there is an $\e$-independent $c>0$ such that $J^\e(t,x)\geq c$ for all $(t,x)\in\overline{S}\times\overline{\Omega}$.

We introduce the transformed coefficient functions needed to transform equations~\eqref{p:full_problem_moving:1}-\eqref{p:full_problem_moving:9} in a fixed domain, i.e., without movement of the phase interface:
\begin{subequations}\label{s:ref_coefficients}
\begin{alignat}{2}
	\A^\e&\colon S\times\Omega\to\R^{3\times3\times3\times3},\quad&\A^\e B&=\frac{1}{2}\left((F^{\e})^{-T} B+\left((F^{\e})^{-T} B\right)^T\right),\label{s:ref_coefficients:1}\\
	\CC_i^{ref,\e}&\colon S\times\Omega_i^\e\to\R^{3\times3\times3\times3},\quad& \CC_i^{ref,\e} &=J^\e (\A^\e)^T\CC_i\A^\e,\\
	\alpha_i^{ref,\e}&\colon S\times\Omega_i^\e\to\R^{3\times3}, \quad&\alpha_i^{ref,\e} &=J^\e \alpha_i(F^{\e})^{-T} ,\\
	H_\Gamma^{ref,\e}&\colon S\times\Gamma^\e\to\R^{3\times3}, \quad& H_\Gamma ^{ref,\e} &=J^\e \sigma_0H_\Gamma^\e (F^{\e})^{-1} ,\\
	c_{i}^{ref,\e}&\colon S\times\Omega_i^\e\to\R, \quad& c_{di}^{ref,\e} &=J^\e\rho_ic_{di},\\
	\gamma_i^{ref,\e}&\colon S\times\Omega_i^\e\to\R,\quad&\gamma_i^{ref,\e}&=J^\e\gamma_i(F^{\e})^{-T},\\
	v^{ref,\e}&\colon S\times\Omega\to\R^3, \quad& v^{ref,\e} &=(F^{\e})^{-1} v^\e ,\\
	K_i^{ref,\e}&\colon S\times\Omega_i^\e\to\R^{3\times3}, \quad& K_i^{ref,\e} &=J^\e (F^{\e})^{-1} K_i(F^{\e})^{-T} ,\\
	W_\Gamma^{ref,\e}&\colon S\times\Gamma^\e\to\R, \quad& W_\Gamma^{ref,\e} &=J^\e W_\Gamma^\e,\label{s:ref_coefficients:9}\\
	f_{u_i}^{ref,\e}&\colon S\times\Omega_i^\e\to\R^3, \quad& f_{u_i}^{ref,\e} &=J^\e\widehat{f_{u_i}}^\e,\label{s:ref_coefficients:10}\\
	f_{\theta_i}^{ref,\e}&\colon S\times\Omega_i^\e\to\R,\quad& f_{\theta_i}^{ref,\e} &=J^\e\widehat{f_{\theta_i}}^\e\label{s:ref_coefficients:11}.
\end{alignat}
\end{subequations}
For a given function $v=v(t,x)$, we denote the corresponding transformed function by ${\widehat{v}}^\e(t,x)=v(t,s^\e(t,x))$. 
Then, as a consequence of the estimate~\eqref{s:estimate_movement}, there is a $C>0$ independent of $t\in S$, $x\in\Omega$, and $\e>0$ such that
\begin{subequations}\label{lemma:eq:estimates_movement}
\begin{multline}\label{lemma:eq:estimates_movement:1}
	\left|\CC_i^{ref,\e}\right|+
	\left|\alpha_i^{ref,\e}\right|+
	\e\left|H_\Gamma^{ref,\e}\right|+
	\left|c_{i}^{ref,\e}\right|\\
	+\e^{-1}\left|v^{ref,\e}\right|+
	\left|K_i^{ref,\e}\right|+
	\e^{-1}\left|W_\Gamma^{ref,\e}\right|
	\leq C.
\end{multline}
Furthermore, using the uniform positivity of $J^\e$, we get the following uniform (all independent of $t\in S$, $x\in\Omega$, and $\e>0$) positivity estimates
\begin{alignat}{2}
	\CC_i^{ref,\e}M:M&\geq c\left|M\right|^2&\quad&\text{for all}\ \ M\in\Sym(3),\label{lemma:eq:estimates_movement:2}\\
	\alpha_i^{ref,\e}v\cdot v&\geq c\left|v\right|^2&\quad&\text{for all}\ \ v\in\R^3,\label{lemma:eq:estimates_movement:3}\\
	c_{i}^{ref,\e}&\geq c,&\label{lemma:eq:estimates_movement:4}\\
	K_i^{ref,\e}v\cdot v&\geq c\left|v\right|^2&\quad&\text{for all}\ \ v\in\R^3.\label{lemma:eq:estimates_movement:5}
\end{alignat}
\end{subequations}

Taking the back-transformed quantities (defined on the initial periodic domains $\Omega_i^\e$) $U_i^\e\colon S\times\Omega_i^\e\to\R^3$ and $\Theta_i^\e\colon S\times\Omega_i^\e\to\R^3$ given via $U_i^\e(t,x)=u_i^\e(t,s^{-1,\e}(t,x))$ and $\Theta_i^\e(t,x)=\theta_i^\e(t,s^{-1,\e}(t,x))$,\footnote{%
Here, $s^{-1,\e}\colon \overline{S}\times\overline{\Omega}\to\overline{\Omega}$ is the inverse function of $s^\e$.} we get the following problem in fixed coordinates (for more details regarding the transformation to a fixed domain, we refer to~\cite{D12,M08, PSZ13}):
\begin{subequations}\label{p:ref}
\begin{alignat}{2} 
	-\dive\left(\CC_A^{ref,\e}(U_A^\e)-\Theta_A^\e\alpha_A^{ref,\e}\right)&=f_{u_A}^{ref,\e} \quad &&\text{in}\ \ S\times\Omega_A^\e,\label{p:ref:1}\\
	-\dive\left(\e^2\CC_B^{ref,\e} e(U_B^\e)-\e\Theta_B^\e\alpha_B^{ref,\e}\right)&=f_{u_B}^{ref,\e} \quad &&\text{in}\ \ S\times\Omega_B^\e,
\end{alignat}
\begin{alignat}{2}
	\begin{split}
	\partial_t\left(c_A^{ref,\e}\Theta_A^\e+\gamma_A^{ref,\e}:\nabla U_A^\e\right)-\dive\left(K_A^{ref,\e}\nabla\Theta_A^\e\right)\\
	-\dive\left(\left(c_A^{ref,\e}\Theta_A^\e+\gamma_A^{ref,\e}:\nabla U_A^\e\right)v^{ref,\e}\right)&=f_{\theta_A}^{ref,\e}\quad \text{in}\ \ S\times\Omega_A^\e,
	\end{split}\\
	\begin{split}
	\partial_t\left(c_B^{ref,\e}\Theta_B^\e+\e\gamma_B^{ref,\e}:\nabla U_B^\e\right)-\dive\left(\e^2K_B^{ref,\e}\nabla\Theta_B^\e\right)\\
	-\dive\left(\left(c_B^{ref,\e}\Theta_B^\e+\e\gamma_B^{ref,\e}:\nabla U_B^\e\right)v^{ref,\e}\right)
	&=f_{\theta_B}^{ref,\e}\quad \text{in}\ \ S\times\Omega_B^\e,
	\end{split}\label{p:ref:4}
\end{alignat}
\end{subequations}
complemented with interface transmission, boundary, and initial conditions.

Note that the structure of this system is similar to the moving interface problem given via equations~\eqref{p:full_problem_moving:1}-\eqref{p:full_problem_moving:9}, except for the advection terms, some additional non-isotropic effects, and the time/space dependency of all coefficients.

\section{Analysis of the micro problem}\label{section:analysis}
We introduce the functional spaces
\begin{align*}
	V_u:=W_0^{1,2}(\Omega)^3,\quad V_\theta:=W^{1,2}(\Omega),\quad H&:=L^2(\Omega)
\end{align*}
and get, after identifying $H$ with their dual via \emph{Riesz's representation map}, the \emph{Gelfand triple}, $V_\theta\hookrightarrow H\hookrightarrow {V_\theta}'$.
With $(,)_H$ and $\langle,\rangle_{V'V}$, we denote the inner product of a Hilbert space $H$ and the dual product of a Banach space $V$, respectively.

Using the well-known \emph{Korn inequality} (see, e.g.,~\cite{DL76}), we see that we can use 
$$
	\|u\|_{V_u}:=\|e(v)\|_{L^2(\Omega_A^\e)^{3\times3}}+\e\|e(v)\|_{L^2(\Omega_B^\e)^{3\times3}},
$$
where $e(u)=1/2(\nabla u+\nabla u^T)$, instead of the standard Sobolev norm for $V_u$.
In the following Lemma, we establish some control on the parameter $\e$

\begin{lemma}\label{lemma:korn}
There is a $C>0$ independent of $\e>0$ such that
\begin{multline}
	\|v\|_{L^2(\Omega)^3}+\|\nabla v\|_{L^2(\Omega_A^\e)^3}+\e\|\nabla v\|_{L^2(\Omega_B^\e)^{3\times3}}\\
	\leq C\left(\|e(v)\|_{L^2(\Omega_A^\e)^{3\times3}}+\e\|e(v)\|_{L^2(\Omega_B^\e)^{3\times3}}\right), \quad v\in V_u.
\end{multline}
\end{lemma}
\begin{proof}
Let $v\in V_u$ and set $v_i^\e:=v_{|_{\Omega_i^\e}}$, $i\in\{A,B\}$.
Then, via extending $v_A^\e$ appropriately to the whole of $\Omega$ (we refer to, e.g.,~\cite[Chapter 1.4]{OSY92}), we call that respective extension $\widetilde{v_A^\e}$, and then using Korn's inequality, we get the $\e$-independent estimate
$$
	\|\widetilde{v_A^\e}\|_{W^{1,2}(\Omega)}\leq C_A\|e(v_A^\e)\|_{L^2(\Omega_A^\e)^{3\times3}}.
$$
Now, since $v\in V_u$, we have $w^\e:=(v-\widetilde{v_A^\e})_{|_{\Omega_B^\e}}\in W^{1,2}_0(\Omega_B^\e)$.
Via a scaling argument (and using Korn's inequality for functions in $ W^{1,2}_0(Y_B)$), we see that
$$
	\|w^\e\|_{L^2(\Omega_B^\e)^3}+\e\|\nabla w^\e\|_{L^2(\Omega_B^\e)^{3\times3}}\leq\e C_B\|e(w^\e)\|_{L^2(\Omega_B^\e)^{3\times3}},
$$
which leads to
\begin{align*}
	\|v_B^\e\|_{L^2(\Omega_B^\e)^3}+\e\|\nabla v_B^\e\|_{L^2(\Omega_B^\e)^{3\times3}}
	&\leq (1+C_B)\left(\|\widetilde{v_A^\e}\|_{W^{1,2}(\Omega)^3}+\e C_B\|e(v_B^\e)\|_{L^2(\Omega_B^\e)^{3\times3}}\right).
\end{align*}
Finally, setting $C=\max\{C_A(2+C_B),C_B\}$, we get the desired estimate.
\end{proof}

Our concept of weak formulation corresponding to the problem in fixed domain, equations~\eqref{p:ref:1}-\eqref{p:ref:4}, is given as: Find $(U^\e,\Theta^\e)\in L^2(S;V_u\times V_\theta)$ such that $\partial_t(U^\e,\Theta^\e)\in L^2(S;V_u'\times V_\theta')$ and $\Theta^\e(0)=\theta^\e_{ext}$ satisfying
\begin{subequations}\label{p:ref_weak}
\begin{multline}\label{p:ref_weak:1}
	\int_{\Omega_A^\e}\CC_A^{ref,\e}e(U_A^\e):e(v)\dx
	+\e^2\int_{\Omega_B^\e}\CC_B^{ref,\e}e(U_B^\e):e(v)\dx\\
	-\int_{\Omega_A^\e}\Theta_A^\e\alpha_A^{ref,\e}:\nabla v \dx
	-\e\int_{\Omega_B^\e}\Theta_B^\e\alpha_B^{ref,\e}:\nabla v \dx\\
	=\int_{\Omega_A^\e}f_{u_A}^{ref,\e}\cdot v\dx
	+\int_{\Omega_B^\e}f_{u_B}^{ref,\e}\cdot v\dx
	+\e^2\int_{\Gamma^\e}H_\Gamma^{ref,\e}n_0^\e\cdot v\ds,
\end{multline}
\begin{multline}\label{p:ref_weak:2}
	\int_S\int_{\Omega_A^\e}\partial_t\left(c_{A}^{ref,\e}\Theta_A^\e\right) v_\theta\dx\dt
	+\int_S\int_{\Omega_A^\e}v^{ref,\e}\Theta_A^\e\cdot\nabla  v_\theta\dx\dt\\
	+\int_S\int_{\Omega_B^\e}\partial_t\left(c_{B}^{ref,\e}\Theta_B^\e\right) v_\theta\dx\dt
	+\int_S\int_{\Omega_B^\e}v^{ref,\e}\Theta_B^\e\cdot\nabla v_\theta\dx\dt\\
	+\int_S\int_{\Omega_A^\e}\partial_{t}\left(\gamma_A^{ref,\e}:\nabla U_A^\e\right)v_\theta\dx\dt
	+\int_S\int_{\Omega_A^\e}v^{ref,\e}\left(\gamma_A^{ref,\e}:\nabla U_A^\e\right)\cdot\nabla v_\theta\dx\dt\\
	+\e\int_S\int_{\Omega_B^\e}\partial_{t}\left(\gamma_B^{ref,\e}:\nabla U_B^\e\right)v_\theta\dx\dt
	+\e\int_S\int_{\Omega_B^\e}v^{ref,\e}\left(\gamma_B^{ref,\e}:\nabla U_B^\e\right)\cdot\nabla v_\theta\dx\dt\\
	+\int_S\int_{\Omega_A^\e}K_A^{ref,\e}\nabla\Theta_A^\e\cdot\nabla v_\theta\dx\dt	
	+\e^2\int_S\int_{\Omega_B^\e}K_B^{ref,\e}\nabla\Theta_B^\e\cdot\nabla v_\theta\dx\dt\\
	+\int_S\int_{\Gamma^\e}W_\Gamma^{ref,\e} v_\theta\ds\dt
	=\int_S\int_{\Omega_A^\e}f_{\theta_A}^{ref,\e}v_\theta\dx\dt
	+\int_S\int_{\Omega_B^\e}f_{\theta_B}^{ref,\e}v_\theta\dx\dt
\end{multline}
\end{subequations}
for all $(v_u,v_\theta)\in L^2(S;V_u\times V_\theta)$.

We start off with the mechanical part, i.e., equation~\eqref{p:ref_weak:1}, and define, for $t\in S$, the linear operators
\begin{alignat*}{2}
	E^\e(t)&\colon V_u\to {V_u}',&\qquad F^\e(t)&\in L^2(\Omega),\\
	e_{th}^\e(t)&\colon H\to {V_u}',&\qquad \mathcal{H}^\e(t)&\in{V_u}'
\end{alignat*}
via
\begin{align*}
	\left\langle E^\e(t)u,v\right\rangle_{{V_u}'V_u}&=\int_{\Omega_A^\e}\CC_A^{ref,\e}(t)e(u):e(v)\dx+\e^2\int_{\Omega_B^\e}\CC_B^{ref,\e}(t)e(u):e(v)\dx,\\
	\left\langle e_{th}^\e(t)\p,v\right\rangle_{{V_u}'V_u}&=\int_{\Omega_A^\e}\p\alpha_A^{ref,\e}(t):\nabla v \dx+\e\int_{\Omega_B^\e}\p\alpha_B^{ref,\e}(t):\nabla v \dx,\\
	F_u^\e(t)&=\begin{cases} f_{u_A}^{ref,\e}(t),\quad &x\in\Omega_A^\e\\ f_{u_B}^{ref,\e}(t),\quad &x\in\Omega_B^\e\end{cases},\\
	\left\langle \mathcal{H}^\e(t),v\right\rangle_{{V_u}'V_u}&=\e^2\int_{\Gamma^\e}H_\Gamma^{ref,\e}(t)n_0^\e\cdot v\ds.
\end{align*}
The weak form~\eqref{p:ref_weak:1} is then equivalent to the operator equation
\begin{equation}\label{problem:mechanic_operator}
	E^\e(t)U^\e-e_{th}^\e(t)\Theta^\e=F_u^\e(t)+\mathcal{H}^\e(t)\quad\text{in}\ \ V_u'.
\end{equation}
\begin{lemma}\label{a:mech_op}
The operator $E^\e(t)$, $t\in\overline{S}$, is coercive, continuous (both uniformly in time and in the parameter $\e$), and symmetric.\\
\end{lemma}
\begin{proof}
Let $u,v\in H^1(\Omega)$.
Due to the estimates~\eqref{lemma:eq:estimates_movement:2},~\eqref{lemma:eq:estimates_movement:3}, we instantly have
\begin{align*}
	\left\langle E^\e(t)v,v\right\rangle_{{V_u}'V_u}
		&\geq c_\CC \left\|v\right\|_{V_u}^2,\\
	\Big|\left\langle E^\e(t)u,v\right\rangle_{{V_u}'V_u}\Big|
		&\leq C_\CC\left\|u\right\|_{V_u}\left\|v\right\|_{V_u}	
\end{align*}
for all $t\in\overline{S}$ and all $\e>0$.
The symmetry follows, after transforming to moving coordinates, from the symmetry of the $\CC_i$.
\end{proof}
Since $E^\e(t)$ is coercive and continuous, we therefore established, via \emph{Lax-Milgram's Lemma}, that, for all $f_{u}^{ref,\e}\in L^2(\Omega)$, $H_\Gamma^{ref,\e}\in L^2(\Gamma^\e)$, $\Theta^\e\in L^2(\Omega)$, and $t\in \overline{S}$, there is a unique weak solution $U^\e(t)\in V_u$ to Problem~\ref{problem:mechanic_operator}.
In particular, the inverse operator $E^{-1,\e}(t)$ is well defined and also linear, bounded, and coercive.
 
Now, we want to turn our attention to the heat related part of our system, i.e., equation~\eqref{p:ref_weak:2}, where we deal with the coupling\footnote{And therefore the mixed derivative term for the deformations $u_i^\e$.} due to the dissipation term.
This is done by combining the structure of the full problem (the \emph{coupling operators} are basically dual to one another) and the just investigated properties of the operators of the mechanical part.  
Note that the following considerations regarding the \emph{thermal stress operator} $e_{th}^\e$ are (in spirit) quite similar to those presented in~\cite{S02}. 
We see that, for $v_\theta\in V_\theta$ and $v_u\in V_u$, we have\footnote{Here, we used that $\jump{\alpha^{ref,\e}(t)}=\jump{\alpha}$ and $\dive(J^\e F^\e)=0$.}
\begin{align*}
	\left\langle e_{th}^\e(t)v_\theta,v\right\rangle_{{V_u}'V_u}
	&=\int_{\Omega_A^\e}v_\theta\alpha_A^{ref,\e}(t):\nabla v_u \dx+\e\int_{\Omega_B^\e} v_\theta\alpha_B^{ref,\e}(t):\nabla v_u \dx\\
	\begin{split}
		&=-\int_{\Omega_A^\e}\alpha_A^{ref,\e}(t)\nabla v_\theta\cdot v_u\dx-\e\int_{\Omega_B^\e}\alpha_B^{ref,\e}(t)\nabla v_\theta\cdot v_u\dx\\
		&\hspace{2.5cm}+\int_{\Gamma^\e}\jump{\alpha}v_\theta n^\e(t)\cdot v_u\ds,
	\end{split}
\end{align*}
and, as a result,
$$
	e_{th}^\e(t)_{| V_\theta}\colon V_\theta\to L^2(\Omega)^3\times L^2(\Gamma^\e)^3\subset{V_u}'.
$$ 
In addition, we take a look at the corresponding dual operator $\left(e_{th}^\e(t)_{| V_\theta}\right)'\colon L^2(\Omega)^3\times L^2(\Gamma^\e)^3\to {V_\theta}'$ given via
$$
	\left\langle\left(e_{th}^\e(t)_{| V_\theta}\right)'[f,g],v_\theta\right\rangle_{{V_\theta}'V_\theta}=\left(e_{th}^\e(t)_{| V_\theta}v_\theta,[f,g]\right)_{L^2(\Omega)\times L^2(\Gamma^\e)}.
$$
For functions $v_u\in V_u$, we have $v_u=[v_u,{v_u}_{|\Gamma^\e}]\in L^2(\Omega)^3\times L^2(\Gamma^\e)^3$ and see that
\begin{align*}
	\left\langle\left(e_{th}^\e(t)_{| V_\theta}\right)'v_u,v_\theta\right\rangle_{{V_\theta}'V_\theta}
		&=\int_{\Omega_A^\e}v_\theta\alpha_A^{ref,\e}(t):\nabla v_u \dx+\int_{\Omega_B^\e}v_\theta\alpha_B^{ref,\e}(t):\nabla v_u \dx.
\end{align*}
As a consequence, we have $\left(e_{th}^\e(t)_{| V_\theta}\right)'_{|V_u}\colon V_u\to H\subset {V_\theta}'$.
For $v_u\in V_u$ and $f\in H$,
\begin{align*}
	\left(\left(e_{th}^\e(t)_{| V_\theta}\right)'_{|V_u}v_u,f\right)_H
		&=\int_{\Omega_A^\e}f\alpha_A^{ref,\e}(t):\nabla v_u \dx+\int_{\Omega_B^\e}f\alpha_B^{ref,\e}(t):\nabla v_u \dx\\
		&=\left\langle e_{th}^\e(t)f,v_u\right\rangle_{{V_u}'V_u},
\end{align*}
which implies $\left(e_{th}^\e(t)_{| V_\theta}\right)'_{|V_u}=\left(e_{th}^\e(t)\right)'$.

Following from the definition of the operator $e_{th}^\e(t)$, we immediately have the following uniform estimate:
\begin{lemma}\label{lemma:coupling_operators}
For $v_u\in V_u$ and $f\in H$, it holds (uniform in $t\in\overline{S}$ and $\e>0$)
\begin{align}
	\left|\left\langle e_{th}^\e(t)f,v_u\right\rangle_{{V_u}'V_u}\right|&\leq C\|f\|_H\left(\|\nabla v_u\|_{L^2(\Omega_A^\e)^{3\times3}}+\e\|\nabla v_u\|_{L^2(\Omega_B^\e)^{3\times3}}\right).
\end{align}
\end{lemma}
Now, we go on introducing the following linear, $t$-parametrized functions\footnote{Note that $\gamma^{ref,\e}=\frac{\gamma}{\alpha}\alpha^{ref,\e}$.}
\begin{alignat*}{2}
	B_1^\e(t)&\colon H\to H,&\quad B_1^\e(t)f&=c^{ref,\e}(t)f,\\
	B_2^\e(t)&\colon H\to H,&\quad B_2^\e(t)f&=\frac{\gamma}{\alpha}\left({e_{th}^\e(t)}\right)'E^{-1,\e}(t) e_{th}^\e(t)f,\\
	A_1^\e(t)&\colon V_\theta\to{V_\theta}',&\quad \left\langle A_1^\e(t)v_\theta,v_\theta\right\rangle_{{V_\theta}'V_\theta}&=\left(\rho c_dv^{ref,\e}(t)v_\theta,\nabla v_\theta\right)_H\\
	A_2^\e(t)&\colon V_\theta\to {V_\theta}',&\quad\langle A_2^\e(t)v_\theta,v_\theta\rangle_{{V_\theta}'V_\theta}&=\left(K^{ref,\e}(t)\nabla v_\theta,\nabla v_\theta\right)_{H},\\
	A_3^\e(t)&\colon V_\theta\to {V_\theta}',&\quad\left\langle A_3^\e(t)v_\theta,v_\theta\right\rangle_{{V_\theta}'V_\theta}&=\left(v^{ref,\e}(t)B_2(t)v_\theta,\nabla v_\theta\right)_H.
\end{alignat*}
and $\mathcal{F}_u^\e(t)\in {V_\theta}'$, $F_\theta^\e(t)\in L^2(S;H)$ via
\begin{align*}
	\left\langle \mathcal{F}_u^\e(t),v_\theta\right\rangle_{{V_\theta}'V_\theta}&=\left(\partial_t\left(\frac{\gamma}{\alpha}\left({e_{th}^\e(t)}\right)'E^{-1,\e}\left(F_u^\e(t)+\mathcal{H}^\e(t)\right)\right),v_\theta\right)_H\\
	&\hspace{1cm}+\left(v^{ref,\e}\frac{\gamma}{\alpha}\left({e_{th}^\e(t)}\right)'E^{-1,\e}\left(F_u^\e(t)+\mathcal{H}^\e(t)\right),\nabla v_\theta\right)_H,\\
	F_\theta^\e(t)&=\begin{cases} f_{\theta_A}^{ref,\e}(t),\quad &x\in\Omega_A^\e\\ f_{\theta_B}^{ref,\e}(t),\quad &x\in\Omega_B^\e\end{cases}.
\end{align*}
We note that $\mathcal{F}_u^\e(t)$ is well defined if, for example, $F_u^\e\in C^1(S;H)^3$.

The variational formulation~\eqref{p:ref_weak:2} can then be rewritten as: Find $\Theta^\e\in L^2(S;V_\theta)$ such that $\partial_t\Theta^\e\in L^2(S;{V_\theta}')$, such that $\Theta^\e(0)=\theta_{0}^\e$, and such that
\begin{equation}\label{a:operator_heat}
	\partial_t\left(\sum_{i=1}^2B_i^\e(t)\Theta^\e\right)+\sum_{i=1}^3A_i^\e(t)\Theta^\e+\mathcal{W}_\Gamma^\e(t)=F_\theta^\e(t)-\mathcal{F}_u^\e(t)\quad\text{in}\ \ V_\theta'.
\end{equation}
\begin{lemma}\label{lemma:time_operator}
The operator $B_2^\e$ is continuous (uniformly in $t\in\overline{S}$ and $\e>0$), self-adjoint, and strictly monotone.
In addition, for every $f,g\in H$, we have $(B_2^\e(\cdot)f,g)_H\in L^\infty(S)$.
\end{lemma}
\begin{proof}
We start off with proving the continuity property.
Let $f\in L^2(\Omega)$ and $u_f^\e:=E^{-1,\e}(t) e_{th}^\e(t)f$, i.e., the unique solution of
\begin{align*}
	\left\langle E^\e(t)u_f^\e,v_u\right\rangle_{{V_u}'V_u}=(e_{th}^\e(t)f,v_u)_H, \quad v_u\in V_u.
\end{align*}
Due to the estimates from Lemmas~\ref{a:mech_op} and~\ref{lemma:coupling_operators}, we instantly have $\|u_f^\e\|_{V_u}\leq C\|f\|_H$, which implies, for all $g\in H$,
\begin{align*}
	\left|\left(B_2^\e f,g\right)_H\right|\leq\frac{\gamma}{\alpha}\|e_{th}^\e(t)g\|_{{V_u}'}\|u_f^\e\|_{V_u}\leq C\|g\|_{H}\|f\|_{H},
\end{align*}
where $C>0$ is independent of both $t\in\overline{S}$ and $\e>0$.
As an immediate consequence, $(B_2^\e(\cdot)f,g)_H\in L^\infty(S)$.
Furthermore, since
\begin{align*}
	\left(B_2^\e f,g\right)_\Omega,
		&=\left\langle e_{th}^\e(t)g,E^{-1,\e}(t) e_{th}^\e(t)f\right\rangle_{{V_u}'V_u}
\end{align*}
and since $E^{-1,\e}$ is strictly monotone and symmetric, we also have that $B_2^\e$ is monotone and self-adjoint. 
\end{proof}
We establish some further regularity (w.r.t.~time) of the following operator:
\begin{align*}
	B^\e(t)\colon L^2(\Omega)\to L^2(\Omega)\quad\text{via}\quad B^\e(t)=B_1^\e(t)+B_2^\e(t).
\end{align*}
\begin{lemma}\label{lemma:time_regularity}
There is a $C>0$ independent of $t\in S$ and $\e>0$ such that 
$$
	\left|\ddt\left(B^\e(t)f,g\right)_H\right|\leq C\left\|f\right\|_H\left\|g\right\|_H
$$
for all $f,g\in H$.
\end{lemma}
\begin{proof}
Let $f, g\in H$ be given.
Then,
\begin{align*}
	\left|\ddt\left(B^\e(t)f,g\right)_H\right|
		\leq\left|\partial_t\left(c^{ref,\e}(t)\right)\right|\left\|f\right\|_H\left\|g\right\|_H+\left|\ddt\left(\left({e_{th}^\e(t)}\right)'E^{-1,\e}(t) e_{th}^\e(t) f,g\right)_H\right|.
\end{align*}
In addition,
\begin{align*}
\left(\left({e_{th}^\e(t)}\right)'E^{-1,\e}(t) e_{th}^\e(t) f,g\right)_H
	=\left\langle e_{th}^\e(t) g,E^{-1,\e}(t) e_{th}^\e(t) f\right\rangle_{{V_u}'V_u},
\end{align*}
where $u_{f}^\e(t):=E^{-1,\e}(t) e_{th}^\e(t)f$ admits the $\e$-uniform bound
$$
	\|u_f^\e\|_{V_u}\leq C\|f\|_H.
$$
Formally, provided that all derivatives exist, we have
$$
	\partial_t u_{f}^\e(t)=\partial_tE^{-1,\e}(t) e_{th}^\e(t)f+E^{-1,\e}(t)\partial_te_{th}^\e(t)f.
$$
Introducing the operators $\widetilde{E^\e}(t)\colon V_u\to {V_u}'$ and $\widetilde{e_{th}^\e}(t)\colon H\to {V_u}'$ via
\begin{align*}
	\left\langle \widetilde{E^\e}(t)u,v\right\rangle_{{V_u}'V_u}&=\int_{\Omega_A^\e}\partial_t\CC_A^{ref,\e}(t)e(u):e(v)\dx+\e^2\int_{\Omega_B^\e}\partial_t\CC_B^{ref,\e}(t)e(u):e(v)\dx,\\
	\left\langle \widetilde{e_{th}^\e}(t)f,v\right\rangle_{{V_u}'V_u}&=\int_{\Omega_A^\e}f\partial_t\alpha_A^{ref,\e}(t):\nabla v \dx+\e\int_{\Omega_B^\e}f\partial_t\alpha_B^{ref,\e}(t):\nabla v \dx,
\end{align*}
and $\widetilde{u_{f}^\e}\in V_u$ as the unique solution to
\begin{equation}\label{eq:tderiv}
	\left\langle E^\e(t)\widetilde{u_{f}^\e},v_u\right\rangle_{{V_u}'V_u}=\left\langle \widetilde{e_{th}^\e}{f},v_u\right\rangle_{{V_u}'V_u}-\left\langle \widetilde{E^\e}u_{f}^\e,v_u\right\rangle_{{V_u}'V_u},\quad v_u\in V_u,
\end{equation}
we see that this is justified and $\partial_tu_{f}^\e=\widetilde{u_{f}^\e}$. 
Furthermore, via testing equation~\eqref{eq:tderiv} with $\partial_tu_{f}^\e$ and using both the uniform bounds on the coefficients and the estimate on $u_f^\e$, inequality~\eqref{lemma:eq:estimates_movement:1}, we see that 
$$
	\left\|\partial_tu_{f}^\e(t)\right\|_{V_u}\leq C\left\|f\right\|_H,
$$
where $C>0$ is independent of $t\in\overline{S}$ and $\e>0$, and, due to 
\begin{align*}
	\ddt\left(B_2^\e(t) f,g\right)_H
		&=\int_{\Omega_A^\e}g\partial_t\alpha_A^{ref,\e}(t):\nabla u_{f}^\e(t)\dx+\int_{\Omega_A^\e}g\alpha_A^{ref,\e}(t):\nabla \partial_tu_{f}^\e(t) \dx\\
		&\hspace{0.8cm}+\e\int_{\Omega_B^\e}g\partial_t\alpha_B^{ref,\e}(t):\nabla u_{f}^\e(t)\dx+\e\int_{\Omega_B^\e}g\alpha_B^{ref,\e}(t):\nabla\partial_tu_{f}^\e(t) \dx,
\end{align*}
we then get the proposed estimate.
\end{proof}

We introduce the operator $A^\e(t)\colon V_\theta\to{V_\theta}'$ via $A^\e(t)=\sum_{i=1}^3A_i^\e(t)$.
\begin{lemma}\label{a:lemma:coercivity}
There are $\lambda_1,\lambda_2>0$ (independent of $t\in\overline{S}$ and $\e>0$) such that
$$
	\left\langle A^\e(t)v_\theta,v_\theta\right\rangle_{{V_\theta}'V_\theta}+\lambda_1\left(B^\e(t)v_\theta,v_\theta\right)_H\geq\lambda_2\left\|v_\theta\right\|_{V_\theta},\quad v_\theta\in V_\theta.
$$
\end{lemma}
\begin{proof}
Let $v_\theta\in V_\theta$.
Due to the positivity of $c^{ref,\e}$, equation~\eqref{lemma:eq:estimates_movement:4}, and the strict monotonicity of $B_2(t)$, cf.~Lemma~\ref{lemma:time_operator}, we have
\begin{align*}
	\left(B^\e(t)v_\theta,v_\theta\right)_H\geq c\left\|v_\theta\right\|^2_H,\quad v_\theta\in V_\theta.
\end{align*}
Using the positivity of $K_i$~\eqref{lemma:eq:estimates_movement:5}, the boundedness of $\e^{-1}|v^{ref}|$~\eqref{s:ref_coefficients}, and the continuity estimate for $B_2^\e$ established in Lemma~\ref{lemma:time_operator}, we get
\begin{align*}
	\left\langle A^\e(t)v_\theta,v_\theta\right\rangle_{{V_\theta}'V_\theta}
		&\geq C_1\left(\left\|\nabla v_\theta\right\|^2_{L^2(\Omega_A^\e)}+\e^2\left\|\nabla v_\theta\right\|^2_{L^2(\Omega_B^\e)}\right)-C_2\left\|v_\theta\right\|_H^2
\end{align*}
From those estimates, we see that the statement holds.
\end{proof}
Having now these results available, we are finally able to prove the main existence theorem for the coupled thermoelasticity problem formulated in fixed coordinates.
\begin{theorem}[Existence Theorem]
Let $F_u^\e\in C^1(S;H)^3$, $F_\theta^\e\in L^2(S\times\Omega)$, and $\theta_0^\e\in L^2(\Omega)$.
Then, there exists a unique $(U^\e,\Theta^\e)\in L^2(S;V_u\times V_\theta)$ such that $\partial_t(U^\e,\Theta^\e)\in L^2(S;{V_u}'\times {V_\theta}')$, such that $\Theta^\e(0)=\theta_{0}^\e$ solving the variational system~\eqref{p:ref_weak} for fixed coordinates.
\end{theorem}
\begin{proof}
In light of the coercivity-type estimate established in Lemma~\ref{a:lemma:coercivity} and the continuity estimate for $B_2'(t)$ from Lemma~\ref{lemma:time_regularity}, we see (\cite[Chapter III, Proposition 3.2 and Proposition 3.3]{S96}) that there is  a unique $\Theta^\e\in L^2(S; V_\theta)$ such that $\partial_t\Theta^\e\in L^2(S;{V_\theta}')$, $\Theta^\e(0)=\theta_{0}^\e$, and
$$
	\ddt\left(B^\e(t)\Theta^\e\right)+A^\e(t)\Theta^\e+\mathcal{W}_\Gamma^\e(t)=F_\theta^\e(t)-\mathcal{F}^\e_u(t)\qquad\text{in}\quad {V_\theta}'.
$$
Defining, for every $t\in\overline{S}$, \footnote{Note that, since $\Theta^\e\in C(\overline{S};H)$, this is well-defined.} 
$$
	U^\e(t):=E^{-1,\e}(t)\left(e_{th}^\e(t)\Theta^\e(t)+F_u^\e(t)+\mathcal{H}^\e(t)\right)\in V_u,
$$
we see that $\partial_tU^\e\in L^2({S;V_u}')$ and that $U^\e(t)$ solves the mechanical part given via the variational equation~\eqref{problem:mechanic_operator}.
\end{proof}
Transforming the solution $(U^\e,\Theta^\e)$ back to moving coordinates, i.e., $u^\e(t,x)=U^\e(t,s^\e(t,x))$ and $\theta^\e(t,x)=\Theta^\e(t,s^\e(t,x))$, we then get the solution to the original problem given by equations~\eqref{p:full_problem_moving:1}-\eqref{p:full_problem_moving:9}.
In the following theorem, we establish the {\em a priori} estimates needed to justify the homogenization process.
\begin{theorem}[$\e$-independent {\em a priori} estimates]\label{a:theorem_apriori}
Assuming that
\begin{align*}
	\|F_u^\e\|_{C^1(S;L^2(\Omega))^3}+\|F_\theta^\e\|_{L^2(S\times\Omega)}+\|\theta_0^\e\|_{L^2(\Omega)}\leq C,
\end{align*}
we have
\begin{multline}
	\|\Theta^\e\|_{L^\infty(S;H)}+\|\nabla\Theta^\e\|_{L^2(S\times\Omega_A^\e)^3}+\e\|\nabla\Theta^\e\|_{L^2(S\times\Omega_B^\e)^3}\\
	+\|U^\e\|_{L^\infty(S;H)^3}+\|\nabla U^\e\|_{L^\infty(S;L^2(\Omega_A^\e))^{3\times3}}+\e\|\nabla U^\e\|_{L^\infty(S;L^2(\Omega_B^\e))^{3\times3}}\leq C,
\end{multline}
where $C$ is independent of the choice of $\e$.
\end{theorem}
\begin{proof}
Testing the variational equality~\eqref{a:operator_heat} with $\Theta^\e$, using the identity
\begin{align*}
	\left(\partial_t\left(B^\e(t)v_\theta\right),v_\theta\right)_H
		&=\left(\partial_t\left(B^\e(t)\right)v_\theta,v_\theta\right)_H+\frac{1}{2}\ddt\left(B^\e(t)v_\theta,v_\theta\right)_H,
\end{align*}
and the uniform operator estimates established in Lemmas~\ref{a:mech_op}-\ref{lemma:time_operator} and in Lemma~\ref{a:lemma:coercivity}, we get
\begin{multline*}
	\ddt\left(B^\e(t)\Theta^\e,\Theta^\e\right)_H+\left\|\nabla \Theta^\e\right\|_{L^2(\Omega_A^\e)}^2+\e^2\left\|\nabla \Theta^\e\right\|_{L^2(\Omega_B^\e)}^2\\
		\leq C\left(\left\|\Theta^\e\right\|_H^2+\left\|\mathcal{F}_\theta^\e(t)\right\|^2_H+\left\|\mathcal{F}_u^\e(t)\right\|^2_{{V_\theta}'}+\left\|\mathcal{W}^\e_\Gamma(t)\right\|_{L^2(\Gamma^\e)}\left\|\Theta^\e\right\|_{L^2(\Gamma^\e)}\right).
\end{multline*}
For the temperature on $\Gamma^\e$, we have the following $\e$-trace estimate, see, e.g.,~\cite{A95},
\begin{equation}\label{a:trace_estimate}
	\e\left\|\Theta^\e\right\|_{L^2(\Gamma^\e)}^2\leq C\left(\left\|\Theta^\e\right\|^2_H+\e^2\left\|\nabla\Theta^\e\right\|^2_H\right).
\end{equation}
Integrating over $(0,t)$ and using the positivity of $B_1^\e$ and the monotonicity of $B_2^\e$, we then get
\begin{multline*}
	\left\|\Theta^\e(t)\right\|^2_H+\int_0^t\left\|\nabla \Theta^\e(\tau)\right\|_{L^2(\Omega_A^\e)}^2\dtau+\e^2\int_0^t\left\|\nabla \Theta^\e(\tau)\right\|_{L^2(\Omega_B^\e)}^2\dtau\\
	\leq C\Bigg(\left\|\Theta^\e(0)\right\|^2_H+\int_0^t\left\|\Theta^\e(\tau)\right\|_H^2\dtau+\int_0^t\left\|\mathcal{F}_\theta^\e(\tau)\right\|_H^2\dtau\\
	+\int_0^t\left\|\mathcal{F}_u^\e(\tau)\right\|_{{V_\theta}'}^2\dtau+\int_0^t\left\|\mathcal{W}^\e_\Gamma(\tau)\right\|^2_{L^2(\Gamma^\e)}\dtau\Bigg).
\end{multline*}
A direct application of \emph{Gronwall's inequality} then yields the desired estimates for the temperatures.
Testing equation~\eqref{problem:mechanic_operator} with $U^\e$ and using the trace estimate~\eqref{a:trace_estimate}, we get
\begin{equation*}
	\|U^\e(t)\|^2_{V_u}\leq C\left(\|\Theta^\e(t)\|^2_H
	+\|F^\e(t)\|^2_{H}+\e^2\|H_\Gamma^{ref,\e}(t)\|^2_{L^2(\Gamma^\e)}\right).
\end{equation*}
Via the Korn-type estimate given by Lemma~\ref{lemma:korn}, we see that the estimates for the deformations are valid.
\end{proof}
\section{Homogenization}\label{section:homogenization}
In the following, we use the notion of two-scale convergence to derive a homogenized model.
Our basic references for homogenization, in general, and two-scale convergence, in particular, are \cite{Al92, L02, N89, T09}.
For the convenience of the reader, we recall the definition of two-scale convergence:
\begin{definition}[Two-scale convergence]
A sequence $v^\e\in L^2(S\times\Omega)$ is said to two scale converge two a limit function $v\in L^2(S\times\Omega\times Y)$ ($v^\e\twosc v$) if
\begin{equation*}
	\lim_{\e\to0}\int_S\int_\Omega v^\e(t,x)\p\exe\dx\dt=\int_S\int_\Omega\int_Y v(t,x,y)\p(x,y)\dy\dx\dt
\end{equation*}
for all $\p\in L^2(S\times\Omega;C_\#(Y))$.
\end{definition}
In addition to the two-scale convergence, we introduce the notion of what is sometimes called \emph{strong two-scale convergence}.
This concept is needed to pass to the limit for some products of two-scale convergent sequences.

\begin{definition}[Strong two-scale convergence]
A sequence $v^\e\in L^2(S\times\Omega)$ is said to strongly two scale converge to a limit function $v\in L^2(S\times\Omega\times Y)$ ($v^\e\stwosc u$) if both $v^\e\twosc v$ and
\begin{align*}
\lim_{\e\to0}\|v^\e\|_{L^2(S\times\Omega)}=\|v\|_{L^2(S\times\Omega\times Y)}.
\end{align*}
\end{definition}
It can be shown, see, e.g.,~\cite[Theorem 18]{L02}\footnote{Combined with the remark succeeding the proof of Theorem~18.}, that if $u^\e\twosc u$ and $v^\e\stwosc v$, we then have
$$\int_S\int_{\Omega}u^\e(t,x)v^\e(t,x)\p\left(x,\frac{x}{\e}\right)\di{x}\di{t}\to\int_S\int_\Omega\int_Yu(t,x,y)v(t,x,y)\p(x,y)\di{y}\di{x}\di{t}$$
for all $\p\in C^\infty_0(\Omega;C^\infty_\#(Y))$.

In the following, for a function $v^\e\in\Omega_i^\e$, $i\in\{A,B\}$, we denote its zero extension to the whole of $\Omega$ with $\chi_i^\e v^\e$.
Furthermore, $W^{1,2}_\#(Y)$ is defined as the closure of $C^1_\#(Y)$ w.r.t.~$W^{1,2}$-Norm, and $W^{1,2}_\per(Y_A)$ as the subspace of $W^{1,2}_\#(Y)$ with zero average.
For functions depending on both $x\in\Omega$ and $y\in Y$, we denote derivatives w.r.t.~$y\in Y$ with the subscript $Y$, i.e., $e_Y$, $\nabla_Y$, $\dive_Y$.

By the $\e$-independent estimates established in Theorem~\ref{a:theorem_apriori}, we have the following two-scale limits.
\begin{theorem}[Two-scale limits]\label{h:theorem_twoscale}
There are functions 
\begin{alignat*}{2}
	u_A&\in L^2(S;V_u),\quad&U_B&\in L^2(S\times\Omega;W^{1,2}_\#(Y)^3),\\
	\theta_A&\in L^2(S;V_\theta),\quad&\Theta_B&\in L^2(S\times\Omega;W^{1,2}_\#(Y)),\\
	\widetilde{U_A}&\in L^2(S\times\Omega;W^{1,2}_\#(Y)^3),\quad&\widetilde{\Theta_A}&\in L^2(S\times\Omega;W^{1,2}_\#(Y))
\end{alignat*}
such that
\begin{alignat*}{2}
	\chi_A^\e U_A^\e&\twosc \chi_Au_A,&\quad \chi_A^\e\nabla U_A^\e&\twosc\chi_A\nabla u_A+\chi_A\nabla_Y\widetilde{U_A},\\
	\chi_B^\e U_B^\e&\twosc \chi_BU_B,&\quad \chi_B^\e\nabla U_B^\e&\twosc\chi_B\nabla_YU_B,\\
	\chi_A^\e \Theta_A^\e&\twosc \chi_A\theta_A,&\quad \chi_A^\e\nabla\Theta_A^\e&\twosc\chi_A\nabla \theta_A+\chi_A\nabla_Y\widetilde{\Theta_A},\\
	\chi_B^\e \Theta_B^\e&\twosc\chi_B\Theta_B,&\quad \chi_B^\e\nabla\Theta_B^\e&\twosc\chi_B\nabla_Y\Theta_B.
\end{alignat*}
\end{theorem}
\begin{remark}
Note that, we distinguish between functions that depend on $y\in Y$ and functions independent of $y\in Y$, by using capitalized letters for the former and lowercase letters for the other.

For a function $u=u(t,x,y)$, we set the corresponding transformed function as $\widehat{u}(t,x,y)=u(t,x,s(t,x,y))$.
To keep the notation consistent, we also set $U_B(t,x,y)=u_B(t,x,s(t,x,y))$ and $\Theta_B(t,x,y)=\theta_B(t,x,s(t,x,y))$.
\end{remark}

Now, we introduce the homogenized transformation related quantities (all elements of $L^\infty(S\times\Omega\times Y)$)
\begin{subequations}\label{h:trafo_quantities}
\begin{alignat}{2}
F&\colon\overline{S}\times\overline{\Omega}\times\overline{Y}\to\R^{3\times3},\  &
	F(t,x,y)&:=\nabla_Y s(t,x,y),\\
J&\colon\overline{S}\times\overline{\Omega}\times\overline{Y}\to\R,\quad&
	J(t,x,y)&:=\det\left(\nabla_Y s(t,x,y)\right),\\
v&\colon\overline{S}\times\overline{\Omega}\times\overline{Y}\to\R^{3},\quad&
	v(t,x,y)&:=\partial_ts(t,x,y),\\
W_\Gamma&\colon\overline{S}\times\overline{\Omega}\times\Gamma\to\R,\quad&
	W_\Gamma(t,x,y)&:=v(t,x,y)\cdot n(t,s(t,x,y)),\\
H_\Gamma&\colon\overline{S}\times\overline{\Omega}\times\Gamma\to\R,\quad&
	H_\Gamma(t,x,y)&:=-\dive_Y\left(F^{-1}(t,x,y)n(t,s(t,x,y))\right)
\end{alignat}
\end{subequations}
and see that they are strong two-scale limits of their $\e$-periodic counterpart
\begin{align*}
F^\e\stwosc F,\quad J^\e\stwosc J,\quad \frac{1}{\e}v^\e\stwosc v,\quad\frac{1}{\e}W_\Gamma^\e\stwosc W_\Gamma,\quad\e H_\Gamma^\e\stwosc H_\Gamma.
\end{align*}
This can be seen by using the regularity of the function $s$, the fact that $\e\left[\frac{x}{\e}\right]\to x$, and using~\cite[Lemma 1.3.]{Al92}.\footnote{Note that ignoring the ``mismatch'' $x-\e\left[\frac{x}{\e}\right]$, we basically have $F^\e(t,x)\approx F\left(t,x,\frac{x}{\e}\right)$}.
For a similar situation in the case of peridodic unfolding, we refer to~\cite[Lemma 3.4.6]{D12}.
As a consequence, we also have strong two-scale convergence for the transformed coefficients, see~\eqref{s:ref_coefficients:1}-\eqref{s:ref_coefficients:9}, the limits of whose are labeled via a $ref$-superscript.

We assume that, for $i\in\{A,B\}$ and almost all $t\in S$, there are functions $f_{u_i}(t)$, $f_{\theta_i}(t)$, and $\theta_{i0}\in L^2(\Omega\times Y)$, such that $\widehat{f_{u_i}}\in C^1(S;L^2(\Omega\times Y)^3)$ and $\widehat{f_{u_i}}\in L^2(S\times\Omega\times Y)$, and such that
\begin{alignat*}{2}
\chi_i^\e\widehat{f_{u_i}^\e}^\e\twosc\chi_i\widehat{f_{u_i}},\quad\chi_i^\e\widehat{f_{\theta_i}}^\e\twosc\chi_i\widehat{f_{\theta_i}},\quad\chi_i^\e\theta_{0}^\e\twosc\chi_i\theta_{i0}.
\end{alignat*}
In particular, this implies
\begin{alignat*}{2}
\chi_i^\e f_{u_i}^{ref,\e}\twosc \chi_iJ\widehat{f_{u_i}}=:\chi_if_{u_i}^{ref},\quad\chi_i^\e f_{\theta_i}^{ref,\e}\twosc\chi_iJ\widehat{f_{\theta_i}}=:\chi_if_{\theta_i}^{ref}.
\end{alignat*}

\subsection{Homogenization of the mechanical part}
Let $v_A\in C_0^\infty(\Omega)^3$ and $v_B\in C^\infty(\overline{\Omega};C^\infty_\#(Y))^3$ such that $v_A(x)=v_B(x,y)$ for all $(x,y)\in\Omega\times\Gamma $.
Furthermore, let $\widetilde{v}_A\in C^\infty(\overline{\Omega};C^\infty_\#(Y))^3$.
We introduce functions 
\begin{align*}
v_A^\e&\colon\Omega\to\R^3,\quad v_A^\e(x):=v_A(x)+\e\widetilde{v}_A\left(x,\frac{x}{\e}\right),\\
v_B^\e&\colon\Omega\to\R^3,\quad v_B^\e(x):=v_B\left(x,\frac{x}{\e}\right)+\e\widetilde{v}_A\left(x,\frac{x}{\e}\right),\\
v^\e&\colon\Omega\to\R^3,\quad v^\e(x):=\begin{cases}v^\e_A(x),\ x\in\Omega_A,\\ v^\e_B(x),\ x\in\Omega_B.\end{cases}
\end{align*}
As a consequence, $v^\e\in W_0^{1,2}(\Omega)^3$.
Choosing $v^\e$ as a test function and letting $\e\to0$, we then get, up to a subsequence, the following limit problem:
\begin{multline}\label{h:limit_coupled}
\int_{\Omega}\int_{Y_A}\CC_A^{ref}\big(e(u_A)+e_Y(\widetilde{U}_A)\big):\big(e(v_A)+e_Y(\widetilde{v}_A)\big)\dy\dx\\
	+\int_{\Omega}\int_{Y_B}\CC_B^{ref}e_Y(U_B):e_Y(v_B)\dy\dx\\
	-\int_{\Omega}\int_{Y_A}\alpha_A^{ref}\theta_A:\big(\nabla v_A+\nabla_Y\widetilde{v}_A\big)\dy\dx
	-\int_{\Omega}\int_{Y_B}\alpha_B^{ref}\Theta_B:\nabla_Yv_B\dy\dx\\
	=\int_{\Omega}\int_{Y_A} f_{u_A}\cdot v_A\dy\di{x}
	+\int_{\Omega}\int_{Y_B} f_{u_B}\cdot v_B\dy\di{x}
	+\int_{\Omega}\int_{\Gamma}H_\Gamma^{ref}n\cdot v_A\ds\di{x}
\end{multline}
for all
$$
(v_A,\widetilde{v}_A,v_B)\in C^\infty_0(\Omega)\times C^\infty_0(\Omega;C^\infty_\#(Y))\times C^\infty_0(\Omega;C^\infty_\#(Y)).
$$
such that $v_A(x)=v_B(x,y)$ for all $(x,y)\in\Omega\times\Gamma$.
By density arguments, equation~\eqref{h:limit_coupled} holds also true for all $(v_A,\widetilde{v}_A,v_B)$, where $v_A\in W^{1,2}_0(\Omega)^3$ and $\widetilde{v}_A,v_B\in L^2(\Omega;W^{1,2}_\#(Y))^3$ such that $v_A(x)=v_B(x,y)$ for almost all $(x,y)\in\Omega\times\Gamma$.
As a next step, we are going to decouple the limit problem~\eqref{h:limit_coupled}.
For this goal, we choose $v_A\equiv0$ and $v_B\equiv0$.
We obtain:
\begin{multline}
\int_{\Omega}\int_{Y_A}\CC_A^{ref}\big(e(u_A)+e_Y(\widetilde{U}_A)\big):e_Y(\widetilde{v}_A)\dy\dx\\
	-\int_{\Omega}\int_{Y_A}\alpha_A^{ref}\theta_A:\nabla_Y\widetilde{v}_A\dy\dx
	=0
\end{multline}
for all $\widetilde{v}_A\in L^2(\Omega;W^{1,2}_\#(Y))^3$.

Now, letting $v_A\equiv0$ and forcing $v_B=0$ a.e.~on $\Omega\times\Gamma$, we get
\begin{multline}
\int_{\Omega}\int_{Y_B}\CC_B^{ref}e_Y(U_B):e_Y(v_B)\dy\dx
	-\int_{\Omega}\int_{Y_B}\alpha_B^{ref}\Theta_B:\nabla_Yv_B\dy\dx\\
	=\int_{\Omega}\int_{Y_B} f_{u_B}\cdot v_B\dy\di{x}
\end{multline}
for all $v_B\in L^2(\Omega;W_0^{1,2}(Y_B))^3$.
Next, while keeping $\widetilde{v}_A\equiv0$, we choose test functions such that $v_A(x)=v_B(x,y)$ for almost all $(x,y)\in\Omega\times Y_B$ (in particular, we have that $v_B$ is constant in $y\in Y$) and see that
\begin{multline}
\int_{\Omega}\int_{Y_A}\CC_A^{ref}\big(e(u_A)+e_Y(\widetilde{U}_A)\big):e(v_A)\dy\dx
	-\int_{\Omega}\int_{Y_A}\alpha_A^{ref}\theta_A:\nabla v_A\dy\dx\\
	=\int_{\Omega}\int_{Y_A} f_{u_A}\cdot v_A\dy\di{x}
	+\int_{\Omega}\int_{Y_B} f_{u_B}\cdot v_A\dy\di{x}
	+\int_{\Omega}\int_{\Gamma}H_\Gamma^{ref}n\cdot v_A\ds\di{x}.
\end{multline}
Summarizing, we obtain the following decoupled (with respect to the test functions) system of variational equalities:
\begin{subequations}\label{h:limit_decoupled}
\begin{multline}\label{h:limit_decoupled:1}
\int_{\Omega}\int_{Y_A}\CC_A^{ref}\big(e(u_A)+e_Y(\widetilde{U}_A)\big):e(v_A)\dy\dx
	-\int_{\Omega}\int_{Y_A}\alpha_A^{ref}\theta_A:\nabla v_A\dy\dx\\
	=\int_{\Omega}\int_{Y_A}f_{u_A}\cdot v_A\dy\di{x}
	+\int_{\Omega}\int_{Y_B}f_{u_B}\cdot v_A\dy\di{x}
	+\int_{\Omega}\int_{\Gamma}H_\Gamma^{ref}\overline{n}\cdot v_A\ds\di{x},
\end{multline}
\begin{multline}\label{h:limit_decoupled:2}
\int_{\Omega}\int_{Y_A}\CC_A^{ref}\big(e(u_A)+e_Y(\widetilde{U}_A)\big):e_Y(\widetilde{v}_A)\dy\dx\\
	-\int_{\Omega}\int_{Y_A}\alpha_A^{ref}\theta_A:\nabla_Y\widetilde{v}_A\dy\dx
	=0,
\end{multline}
\begin{multline}\label{h:limit_decoupled:3}
\int_{\Omega}\int_{Y_B}\CC_B^{ref}e_Y(U_B):e_Y(v_B)\dy\dx
	-\int_{\Omega}\int_{Y_B}\alpha_B^{ref}\Theta_B:\nabla_Yv_B\dy\dx\\
	=\int_{\Omega}\int_{Y_B} f_{u_B}\cdot v_B\dy\di{x}
\end{multline}
\end{subequations}
for all $\left(v_A,\widetilde{v}_A,v_B\right)\in W^{1,2}_0(\Omega)^3\times L^2(\Omega;W^{1,2}_\#(Y))^3\times L^2(\Omega;W_0^{1,2}(Y_B))^3$.
In addition to equations~\eqref{h:limit_decoupled:1}-\eqref{h:limit_decoupled:3}, we have the additional constraint $u_A(t,x)=U_B(t,x,y)$ for almost all $(t,x,y)\in S\times\Omega\times\Gamma$.

We go on by introducing cell problems and effective quantities to get a more accessible form of the homogenization limit.
For $j,k\in\{1,2,3\}$ and $y\in Y$, set $d_{jk}=(y_j\delta_{1k},y_j\delta_{2k},y_j\delta_{3k})^T$.
For $t\in S$, $x\in\Omega$, let $\tau^u_{jk}(t,x,\cdot)$, $\tau^u(t,x,\cdot)\in H^1_\per(Y_A)^3$ are the solutions to 
\begin{subequations}
\begin{align}
	0&=\int_{Y_A}\CC_A^{ref}e_Y(\tau^u_{jk}+d_{jk}):e_Y(\widetilde{v}_A)\dy,\\
	0&=\int_{Y_A}\CC_A^{ref}e_Y(\tau^u):e_Y(\widetilde{v}_A)\dy-\int_{Y_A}\alpha_A^{ref}:\nabla_Y\widetilde{v}_A\dy
\end{align}
for all $\widetilde{v}_A\in H^1_\per(Y_A)^3$.
In addition, we introduce the effective elasticity tensor $\CC_A^{eff}\colon S\times\Omega\to\R^{3\times3\times3\times3}$, $\CC_A^{eff}(t,x)=\left(\CC_A^{eff}(t,x)\right)_{1\leq i,j,k,l\leq3}$, via
\begin{equation}
	\left(\CC_A^{eff}\right)_{j_1j_2j_3j_4}=\int_{Y_A}\CC_A^{ref}e_Y\left(\tau_{j_1j_2}^u+d_{j_1j_2}\right) \colon e_Y\left(\tau_{j_3j_4}^u+d_{j_3j_4}\right)\dy.
\end{equation}
Furthermore, we introduce the following effective functions:
\begin{alignat}{2}
	H_{\Gamma}^{eff}&\colon S\times\Omega\to\R,\quad
	H_{\Gamma}^{eff}(t,x)=\int_{\Gamma}H_\Gamma^{ref}(t,x,s) n_0(t,x,s)\ds,\\
	f_u^{eff}&\colon S\times\Omega\to\R,\quad
	f_u^{eff}(t,x)=\int_{Y_A}f_{u_A}^{ref}(t,x,y)\dy+\int_{Y_B}f_{u_B}^{ref}(t,x,y)\dy,\\
	\alpha_A^{eff}&\colon S\times\Omega\to\R^{3\times3},\quad
	\alpha_A^{eff}(t,x)=\int_{Y_A}\left(\alpha_A^{ref}-C_A^{ref}e_Y\big(\tau^u\big)\right)\dy.
	\end{alignat}
\end{subequations}
We see that, at least up to function independent of $y\in Y$, it holds
$$
\widetilde{U_A}(t,x,y)=\sum_{j,k=1}^3\tau_{jk}^u(t,x,y)(e(u_A)(t,x))_{jk}+\tau^u(t,x,y)\theta_A(t,x).
$$
After transforming the microscopic mechanical part to moving coordinates, we are led to
\begin{subequations}\label{h:hom_mech}
\begin{align}
\begin{split}
\int_{\Omega}C_A^{eff}e(u_A):e(v_A)\di{x}
-\int_{\Omega}\alpha_A^{eff}\theta_A:\nabla v_A\dx\hspace{-0.5cm}\\
&=\int_{\Omega}f_u^{eff}\di{x}
	+\int_{\Omega}H_\Gamma^{eff}\di{x},\label{h:hom_mech:1}
	\end{split}\\
\begin{split}
\int_{Y_B(t,x)}\CC_B e_Y(u_B):e_Y(v_B)\dy
	-\int_{Y_B(t,x)}\alpha_B\theta_B\dive_Yv_B\dy\hspace{-2cm}\\
	&=\int_{Y_B(t,x)}f_{u_B}\cdot v_B\dy\label{h:hom_mech:2}
\end{split}
\end{align}
\end{subequations}
for all $v_A\in H_0^1(\Omega)^3$, $v_B\in H_0^1(Y_B(t,x))^3$ and almost all $t\in S$.

\subsection{Homogenization of the heat part}
Let $(v_A,\widetilde{v}_A)\in C^\infty(\overline{S}\times\overline{\Omega})\times C^\infty(\overline{S}\times\overline{\Omega};C^\infty_\#(Y))$ and $v_B\in C^\infty(\overline{S}\times\overline{\Omega};C^\infty_\#(Y))$ such that $v_A(T)=\widetilde{v}_A(T)=v_B(T)=0$ and such that $v_A(t,x)=v_B(t,x,y)$ for all $(t,x,y)\in S\times\Omega\times\Gamma$.
We introduce the functions 
\begin{align*}
v_A^\e&\colon S\times\Omega\to\R^3,\quad v_A^\e(t,x)=v_A(t,x)+\e\widetilde{v}_A\left(t,x,\frac{x}{\e}\right),\\
v_B^\e&\colon S\times\Omega\to\R^3,\quad v_B^\e(t,x)=v_B\left(t,x,\frac{x}{\e}\right)+\e\widetilde{v}_A\left(t,x,\frac{x}{\e}\right),\\
v^\e&\colon S\times\Omega\to\R^3,\quad
v^\e(t,x)=\begin{cases}v^\e_A(t,x),\ x\in\Omega_A,\\ v^\e_B(t,x),\ x\in\Omega_B.\end{cases}
\end{align*}
Then, $v^\e\in W^{1,2}(\Omega)$. 
Choosing $v^\e$ as a test function and letting $\e\to0$, we get, up to a subsequence, the following limit problem:
\begin{multline}
-\int_S\int_\Omega c_{dA}\left|Y_A\right|\theta_A\partial_tv_A\di{x}\dt
	-\int_\Omega c_{dA}\left|Y_A\right|\theta_{A0}\partial_tv_A(0)\di{x}\\
	-\int_S\int_\Omega\int_{Y_B} c_{B}^{ref}\Theta_B\partial_tv_B\dy\di{x}\dt
	-\int_{\Omega}\int_{Y_B}c_{dB}\theta_{B0}v_B(0)\dy\di{x}\\
	+\int_S\int_{\Omega}\int_{Y_B}c_{B}^{ref}v^{ref}\Theta_B\cdot\nabla_Y v_B\dy\di{x}\dt\\
	-\int_S\int_{\Omega}\int_{Y_A}\gamma_A^{ref}:\left(\nabla u_A+\nabla_Y\widetilde{U}_A\right)\partial_{t}v_A\dx\dt
	-\int_S\int_{\Omega}\int_{Y_B}\gamma_B^{ref}:\nabla_Y U_B\partial_{t}v_B\dy\dx\dt\\
	+\int_S\int_{\Omega}\int_{Y_B}v^{ref}\left(\gamma_B^{ref}:\nabla_Y U_B^\e\right)\cdot\nabla_Y v_B\dy\dx\dt\\
	+\int_S\int_{\Omega}\int_{Y_A}K_A^{ref}\left(\nabla\theta_A+\nabla_Y\widetilde{\Theta}_A\right)\cdot\left(\nabla v_A+\nabla_Y\widetilde{v}_A\right)\dy\di{x}\dt\\
	+\int_S\int_{\Omega}\int_{Y_B}K_B^{ref}\nabla_Y\Theta_B\cdot\nabla_Yv_B\dy\di{x}\dt
	+\int_S\int_{\Omega}\int_{\Gamma}W_\Gamma^{ref}v_A\di{s}\di{x}\dt\\
	=\int_S\int_{\Omega}\int_{Y_A}f_{\theta_A}^{ref}v_A\dy\di{x}\dt
	+\int_S\int_{\Omega}\int_{Y_B}f_{\theta_B}^{ref}v_B\dy\di{x}\dt	
\end{multline}
for all $(v_A,\widetilde{v}_A)\in C^\infty(\overline{S}\times\overline{\Omega})\times C^\infty(\overline{S}\times\overline{\Omega};C^\infty_\#(Y))$ and $v_B\in C^\infty(\overline{S}\times\overline{\Omega};C^\infty_\#(Y))$ such that $v_A(T)=\widetilde{v}_A(T)=v_B(T)=0$ and such that $v_A(t,x)=v_B(t,x,y)$ for all $(t,x,y)\in S\times\Omega\times\Gamma$.
Here, $\left|Y_A\right|=\left|Y_A(t,x)\right|$.

Using the same decoupling strategy as for the mechanical part, we obtain the following system of variational equalities:
\begin{subequations}\label{h:hom_decoupled}
\begin{multline}\label{h:hom_decoupled:1}
-\int_S\int_\Omega\rho_Ac_{dA}\left|Y_A\right|\theta_A\partial_tv_A\di{x}\dt
	-\int_\Omega\rho_A c_{dA}\left|Y_A\right|\theta_{A0}\partial_tv_A(0)\di{x}\\
	-\int_S\int_\Omega\left(\int_{Y_B} c_{B}^{ref}\Theta_B\dy\right)\partial_tv_A\di{x}\dt
	-\int_{\Omega}\left(\int_{Y_B}c_{dB}\theta_{B0}\dy\right)v_A(0)\di{x}\\
	-\int_S\int_{\Omega}\left(\int_{Y_A}\gamma_A^{ref}:\left(\nabla u_A+\nabla_Y\widetilde{U}_A\right)\dy+\int_{Y_B}\gamma_B^{ref}:\nabla_Y U_B\dy\right)\partial_{t}v_A\dx\dt\\
	+\int_S\int_{\Omega}\int_{Y_A}K_A^{ref}\left(\nabla\theta_A+\nabla_Y\widetilde{\Theta}_A\right)\cdot\nabla v_A\dy\di{x}\dt
	+\int_S\int_{\Omega}\left(\int_{\Gamma}W_\Gamma^{ref}\di{s}\right)v_A\di{x}\dt\\
	=\int_S\int_{\Omega}\left(\int_{Y_A}f_{\theta_A}^{ref}\dy\right)v_A\di{x}\dt
	+\int_S\int_{\Omega}\left(\int_{Y_B}f_{\theta_B}^{ref}\dy\right)v_A\di{x}\dt,
\end{multline}
\begin{align}\label{h:hom_decoupled:2}
\int_S\int_{\Omega}\int_{Y_A}K_A^{ref}\left(\nabla\theta_A+\nabla_Y\widetilde{\Theta}\right)\cdot\nabla_Y\widetilde{v}_A\dy\di{x}\dt=0,
\end{align}
\begin{multline}\label{h:hom_decoupled:3}
	-\int_S\int_\Omega\int_{Y_B} c_{B}^{ref}\Theta_B\partial_tv_B\dy\di{x}\dt
	-\int_{\Omega}\int_{Y_B}c_{dB}\theta_{B0}v_B(0)\dy\di{x}\\
	+\int_S\int_{\Omega}\int_{Y_B}\rho_Bc_{dB}v^{ref}\Theta_B\cdot\nabla_Yv_B\dy\di{x}\dt\\
	-\int_S\int_{\Omega}\int_{Y_B}\gamma_B^{ref}:\nabla_Y U_B\partial_{t}v_B\dy\dx\dt
	+\int_S\int_{\Omega}\int_{Y_B}v^{ref}\left(\gamma_B^{ref}:\nabla_Y U_B\right)\cdot\nabla_Yv_B\dy\dx\dt\\
	+\int_S\int_{\Omega}\int_{Y_B}K_B^{ref}\nabla_Y\Theta_B\cdot\nabla_Yv_B\dy\di{x}\dt
	=\int_S\int_{\Omega}\int_{Y_B}f_{\theta_B}^{ref}v_B\dy\di{x}\dt
\end{multline}
\end{subequations}
for all $(v_A,\widetilde{v}_A,v_B)\in L^2(S;W^{1,2}(\Omega))\times L^2(S\times\Omega;W_\#^{1,2}(Y))\times L^2(S\times\Omega;W_0^{1,2}(Y_B))$ such that $(\partial_tv_A,\partial_tv_B)\in L^2(S;(W^{1,2}(\Omega)'))\times L^2(S\times\Omega;W^{-1,2}(Y_B))$ and such that $v_A(T)=v_B(T)=0$.

Now, we want to find a more accessible description of the homogenized problem given via equations~\eqref{h:hom_decoupled:1}-\eqref{h:hom_decoupled:3}.
With that in mind, for $j\in\{1,2,3\}$, $t\in S$, $x\in\Omega$, let $\tau_j^\theta(t,x,\cdot)\in H^1_\per(Y_A)$ be the solution to 
\begin{equation}
	\int_{Y_A}K_A^{ref}\left(\nabla_Y\tau_j^\theta+e_j\right)\cdot\nabla_Y\widetilde{v}_A\dy=0,\quad\widetilde{v}_A\in W^{1,2}_\per(Y_A).
\end{equation}
We introduce the following effective functions 
\begin{alignat*}{3}
c^{eff}&\colon S\times\Omega\to\R,&\quad W_{\Gamma}^{eff}&\colon S\times\Omega\to\R,&\quad \gamma_A^{eff}&\colon S\times\Omega\to\R^{3\times3},\\
K_A^{eff}&\colon S\times\Omega\to\R^{3\times3},&\quad f_\theta^{eff}&\colon S\times\Omega\to\R
\end{alignat*}
defined via
\begin{subequations}\label{h:hom_effective}
\begin{align}
c^{eff}(t,x)&=\rho_Ac_{dA}\left|Y_A(t,x)\right|+\alpha_A\int_{Y_A}\dive_Y\left(\tau_1^m\right)(t,x,y)\dy,\\
K_A^{eff}(t,x)_{ij}&=\int_{Y_A}K_A^{ref}(t,x,y)\left(\nabla_Y\tau_j^\theta(t,x,y)+e_j\right)\cdot\left(\nabla_Y\tau_i^\theta(t,x,y)+e_i\right),\\
W_{\Gamma}^{eff}(t,x)&=\int_{\Gamma}W_\Gamma^{ref}(t,x,s)\di{s},\\
f_\theta^{eff}(t,x)&=\int_{Y_A}f_{\theta_A}^{ref}(t,x,y)\dy+\int_{Y_B}f_{\theta_B}^{ref}(t,x,y)\dy\\
\gamma_A^{eff}(t,x)&=\int_Y\left(\gamma_A^{ref}+\gamma_A\nabla_Y\tau_{jk}^{m}(t,x,y)\right)\dy+\gamma_B\left|Y_B(t,x)\right|\mathds{I}_3.
\end{align}
\end{subequations}
The system of variational equalities~\eqref{h:hom_decoupled:1}-\eqref{h:hom_decoupled:3} then reads
\begin{subequations}\label{h:hom_variational}
\begin{multline}\label{h:hom_variational:1}
-\int_S\int_\Omega c^{eff}\theta_A\partial_tv_A\di{x}\dt
	-\int_\Omega c^{eff}(0)\theta_{A0}\partial_tv_A(0)\di{x}\\
	-\int_S\int_\Omega\left(\int_{Y_B(t,x)}\rho_Bc_{dB}\theta_B\dy\right)\partial_tv_A\di{x}\dt
	-\int_{\Omega}\left(\int_{Y_B(0)}\rho_Bc_{dB}\theta_{B0}\dy\right)v_A(0)\di{x}\\
	-\int_S\int_{\Omega}\gamma_A^{eff}:\nabla u_A\partial_{t}v_A\dx\dt
	+\int_S\int_{\Omega}K_A^{eff}\nabla\theta_A\cdot\nabla v_A\dy\di{x}\dt\\
	=-\int_S\int_{\Omega}W_{\Gamma}^{eff}v_A\di{x}\dt
	+\int_S\int_{\Omega}f_\theta^{eff}\di{x}\dt,
\end{multline}
\begin{multline}\label{h:hom_variational:2}
	-\int_S\int_{Y_B} c_{B}^{ref}\Theta_B\partial_tv_B\dy\dt
	+\int_S\int_{Y_B}\rho_Bc_{dB}v^{ref}\Theta_B\cdot\nabla_Yv_B\dy\dt
	-\int_{\Omega}\int_{Y_B}c_{dB}\theta_{B0}v_B(0)\dy\\
	-\int_S\int_{Y_B}\gamma_B^{ref}:\nabla_Y U_B\partial_{t}v_B\dy\dt
	+\int_S\int_{Y_B}v^{ref}\left(\gamma_B^{ref}:\nabla_Y U_B^\e\right)\cdot\nabla_Yv_B\dy\dt\\
	+\int_S\int_{Y_B}K_B^{ref}\nabla_Y\Theta_B\cdot\nabla_Yv_B\dy\di{x}\dt
	=\int_S\int_{Y_B}f_{\theta_B}^{ref}v_B\dy\dt
\end{multline}
\end{subequations}
for all $(v_A,v_B)\in L^2(S;W^{1,2}(\Omega))\times L^2(S\times\Omega;W_0^{1,2}(Y_B))$ such that $(\partial_tv_A,\partial_tv_B)\in L^2(S;(W^{1,2}(\Omega)'))\times L^2(S\times\Omega;W^{-1,2}(Y_B))$ and such that $v_A(T)=v_B(T)=0$.

Finally, we are able to present the complete homogenized problem of the initial highly heterogeneous $\e$-problem given by equations~\eqref{p:full_problem_moving:1}-\eqref{p:full_problem_moving:9}.
We transform the variational equations~\eqref{h:hom_variational:2} to the moving domain formulation and combine the homogenized mechanical system (equations~\eqref{h:hom_mech:1},~\eqref{h:hom_mech:2}) and the homogenized thermo system (equations~\eqref{h:hom_variational:1},~\eqref{h:hom_variational:2}).
Via localization, this results in the following two-scale system of partial differential equations (complemented by initial conditions and macroscopic boundary conditions)
\begin{subequations}
\begin{alignat}{2}
-\dive\left(\CC_A^{eff}e(u_A)-\alpha_A^{eff}\theta_A\right)&=f_u^{eff}+H_\Gamma^{eff}\quad &&\text{in}\ \ S\times\Omega,\label{h:hom_prob:1}\\
\partial_t\Bigg(c^{eff}\theta_A+\rho_Bc_{dB}\int_{Y_B(t,x)}\theta_B\dy+\gamma_A^{eff}:&\nabla u_A\Bigg)\notag\\
-\dive\left(K_A^{eff}\nabla\theta_A\right)&=f_\theta^{eff}-W_{\Gamma}^{eff}\quad &&\text{in}\ \ S\times\Omega\label{h:hom_prob:2},
\end{alignat}
\begin{alignat}{2}
-\dive_Y\left(\CC_B e_Y(u_B)-\alpha_B\theta_B\mathds{I}_3\right)&=f_{u_B}\ \ &&\text{in}\ \ Y_B(t,x)\label{h:hom_prob:3},\\
\rho_Bc_{dB}\partial_t\theta_B+\gamma_B\partial_t\dive_Yu_B-\dive_Y\left(K_B\nabla_Y\theta_B\right)&=f_{\theta_B}\ \ &&\text{in}\ \ Y_B(t,x),\label{h:hom_prob:4}\\
u_B=u_A,\ \ \theta_B&=\theta_A\quad &&\text{on}\ \ \partial Y_B(t,x).\label{h:hom_prob:5}
\end{alignat}
\end{subequations}
This homogenized model is a typical example of what is usually called a \emph{distributed-microstructure} model~\cite{S93}. In simple words this means that  on the one hand, we have obtained an averaged macroscopic description of the coupled thermoelasticity, that is equations~\eqref{h:hom_prob:1} and~\eqref{h:hom_prob:2}, while  on the other hand, these averaged equations are, at every point $x\in\Omega$, additionally coupled with an $x$-parametrized microscopic problem, see equations~\eqref{h:hom_prob:3}-\eqref{h:hom_prob:5}.

The coupling between the two-scales (microscopic and macroscopic), again, is two-fold: a) Via the \emph{Dirichlet}-boundary condition on $\partial Y_B(t,x)$ (equation~\eqref{h:hom_prob:5}), which is a direct consequence of the continuity conditions posed on the phase-interface of the $\e$-microproblem, the macroscopic quantities determine the boundary values of the microscopic quantities.
b) In contrast, in the macroscopic heat equation, we see that the average of the microscopic heat density, i.e.,
$$\rho_Bc_{dB}\int_{Y_B(t,x)}\theta_B\dy$$ 
is part of the overall heat density.
In the case of $\gamma_i=0$, i.e., when there is no dissipation, the \emph{overall effective heat density} $e^{eff}=e^{eff}(t,x)$ would then be given as
$$e^{eff}=c^{eff}\theta_A+\rho_Bc_{dB}\int_{Y_B(t,x)}\theta_B\dy$$
This seems to suggest that equation~\eqref{h:hom_prob:2} should, actually, be interpreted as a balance equation for the so-called overall heat density, where part of the balanced quantity, the microscopic temperature $\theta_B$, is given as a solution to the microscopic heat balance equation.

In the homogenization limit, the phase transformation is a purely microscopic phenomenon, where we have the free boundary $\partial Y_B(t,x)$.
However, the transformation does also turn up in the macroscopic part, where it enters via the volume force densities \emph{effective mean curvature} $H_\Gamma^{eff}$ and the \emph{effective normal velocity} $W_{\Gamma}^{eff}$.

\section*{Acknowledgments}
The authors are indebted to Michael B\"ohm (Bremen) for initiating and supporting this research.
AM thanks NWO MPE ``Theoretical estimates of heat losses in geothermal wells'' (grant nr. 657.014.004) for funding.

\bibliography{literature}{}
\bibliographystyle{plain}

\medskip
\medskip

\end{document}